\documentclass[reqno]{amsart}
\usepackage{amsmath,amsfonts,amssymb,amsthm,dsfont, bm}
\usepackage{float}
\usepackage{enumitem}
\usepackage[usenames,dvipsnames]{xcolor}
\usepackage[colorlinks=true]{hyperref} 
\usepackage{verbatim}
\usepackage{graphicx}
\usepackage{overpic}
\usepackage{mathabx}
\usepackage{mathrsfs}
\usepackage{multirow}
\hypersetup{linkcolor=Violet,citecolor=PineGreen}
\newtheorem{thm}{Theorem}[section]
\newtheorem{lem}[thm]{Lemma}
\newtheorem{prop}[thm]{Proposition}
\newtheorem{cor}[thm]{Corollary}
\theoremstyle{definition}
\newtheorem{defn}[thm]{Definition}
\newtheorem{example}[thm]{Example}
\newtheorem{rmk}[thm]{Remark}
\numberwithin{equation}{section}

\newcommand{\N}{\mathbb{N}}
\newcommand{\R}{\mathbb{R}}

\newcommand{\Z}{\mathbb{Z}}

\newcommand{\A}{\mathcal{A}}
\newcommand{\C}{\mathcal{C}}
\newcommand{\K}{\mathcal{K}}
\newcommand{\M}{\mathscr{M}}

\newcommand{\U}{\mathcal{U}}

\newcommand{\ep}{\varepsilon}
\newcommand\note[1]{\marginpar{\flushleft\sffamily\tiny #1}}

\newcommand{\nin}{{n\in\N}}
\newcommand{\nti}{{n\to\infty}}

\newcommand{\dist}{\mathrm{dist}}

\makeatletter
\newcommand{\setword}[2]{%
  \phantomsection
  #1\def\@currentlabel{\unexpanded{#1}}\label{#2}%
}
\makeatother

\begin{document}
\title[Invariant Sets of  Differential Inclusions -- Boundary Systems]{Invariant Sets and Boundary Systems of Nonautonomous Differential Inclusions}

\author[K.~Kourliouros]{Konstantinos Kourliouros}
\author[I.P.~Longo]{Iacopo P.~Longo}
\author[M.~Rasmussen]{Martin Rasmussen}
\email[Konstantinos Kourliouros]{k.kourliouros10@imperial.ac.uk}
\email[Iacopo P. Longo]{i.longo@exeter.ac.uk}
\email[Martin Rasmussen]{m.rasmussen@imperial.ac.uk}
\address[K.~Kourliouros]{Imperial College London,
Department of Mathematics, 635 Huxley Building, 180 Queen’s Gate
South Kensington Campus, SW7 2AZ London, United Kingdom.}
\address[I.P.~Longo]{University of Exeter,
Department of Mathematics and Statistics, 925 Laver Building, 23 N Park Road, EX4 4QE Exeter, United Kingdom.}
\address[M.~Rasmussen]{Imperial College London,
Department of Mathematics, 637 Huxley Building, 180 Queen’s Gate
South Kensington Campus, SW7 2AZ London, United Kingdom.}


\date{}
\begin{abstract}
In this paper we propose a finite-dimensional and deterministic approach to the study of invariant sets of certain nonautonomous differential inclusions naturally arising in the context of random and control dynamical systems, as well as in systems modeling the dynamical propagation of uncertainty. In particular, to any such differential inclusion, we associate a finite-dimensional and deterministic system of nonautonomous ordinary differential equations, which we call the boundary system, due to its following characteristic property: invariant sets of the differential inclusion lift in a unique way to backward invariant unit normal cones of the associated boundary system, and these are even invariant if the boundary is smooth. We further illustrate the strength of this approach in the study of minimal attractors of nonautonomous linear differential inclusions. Under the natural assumption of  exponential stability for the unperturbed problem, we establish existence and uniqueness of a minimal attractor for the differential inclusion with fiberwise strictly convex, closed, and continuously differentiable boundaries. We also show that the unit normal bundle is in fact the pullback attractor for the skew-product flow associated to the boundary system which extends to the global attractor when the underlying admits a compact base.    

\end{abstract}
\maketitle
\section{Introduction}

In this paper we study (minimal) invariant sets of nonautonomous differential inclusions of the form:
\begin{equation}
    \label{eq:diff-incl}
    \dot{x}\in \overline{B_{\rho}}(f(t,x)),
\end{equation}
where $x$ represents the state of the system, $t$ time, $f:\mathbb{R}\times \mathbb{R}^d\rightarrow \mathbb{R}^d$ is a sufficiently smooth function in the $d$-dimensional Euclidean space $\mathbb{R}^d$, with $d\geq 2$, and $\overline{B_{\rho}}(y)$ represents the closed ball of radius $\rho >0$ centered at the point $y\in \mathbb{R}^d$. 
An invariant set for the differential inclusion \eqref{eq:diff-incl} is a fiber-wise compact non-autonomous set $\M=(M_t)_{t\in \mathbb{R}}$ (i.e.~each $M_t$ is compact in $\mathbb{R}^d$) which is invariant under the two-parameter set-valued evolution operator $\Phi^{t,t_0}$ induced by \eqref{eq:diff-incl}, i.e.~the equality $\Phi^{t,t_0}(M_{t_0})=M_t$ holds for all $t,t_0\in \mathbb{R}$ with $t> t_0$. It is called minimal invariant, if there is no other proper subset of $\M$ which is also invariant by $\Phi^{t,t_0}$.

Differential inclusions of the form \eqref{eq:diff-incl} and their invariant sets $\M$, are fundamental objects of study in various contexts within the theory of dynamical systems and its applications, particularly in control and viability theory (cf.~\cite{aubin1984differential, colonius2012dynamics,lee1986foundations,pontryagin1962optimal}), as well as in the theory of random dynamical systems and the dynamical propagation of uncertainty (cf.~\cite{arnold1998rds,banks2014uncertainty,catuogno2025minimal,chappelle2023rate,cui2017uniform,li2024pullback}) among others. In all these contexts, the set-valued evolution operator $\Phi^{t,t_0}$ induced by the differential inclusion \eqref{eq:diff-incl} commonly models---in terms of so-called reachable sets---the compound evolution of all possible trajectories of the family of differential equations of the form
\begin{equation}
\label{eq:01}
    \dot x = f(t,x) + \rho\xi(t),\quad t\in\R,\, x\in\R^d,
\end{equation}
where the time-dependent random (or deterministic) functions $\xi:\R\to\R^d$  belong to the set $\U=\{\xi\in L^\infty(\R,\R^d)\mid \|\xi\|_\infty\le1\}$ of all noise perturbations or admissible controls.

The (minimal) invariant sets $\M$ of  $\Phi^{t,t_0}$ play a central role in the dynamics of reachable sets, as they commonly represent those regions of the extended phase space (trapping regions) from which trajectories of the differential equation \eqref{eq:01} cannot escape, for any noise or uncertainty realisation (or choice of control, depending on the context) $\xi\in\mathcal{U}$, and for any initial condition from $\M$. In control theory, minimal invariant sets are known as invariant control sets \cite{colonius2012dynamics}, and also as viable sets in viability theory \cite{aubin1984differential}, whereas they appear under possibly different names in various other contexts. 



Several prominent studies have addressed fundamental questions concerning the topological properties of minimal invariant sets of set-valued flows (mostly in the autonomous case), as well as their stability and bifurcations (so-called topological, or hard bifurcations) under parameter variation (c.f.~\cite{gayer2004control1,gayer2005controllability,homburgyoung2010bifrdssurfaces,homburgyounggheraei2013bifrds,lambrasmussen2015topbif,younghomburg2006hard}). While these studies provide a significant advancement in the field and its related areas, they heavily rely on topological (infinite-dimensional) set-valued techniques, and are thus bound by the lack of a Banach space structure on the hyperspace $\mathcal{K}(\mathbb{R}^d)$ of all nonempty compact subsets of the phase space $\mathbb{R}^d$. Questions concerning the subtler differentiable properties of invariant sets, such as the regularity or singularities of their boundaries $\partial \M$, their smooth stability, or the sudden changes in their differentiable structure under parameter variation (smooth bifurcations), remain out of reach inside this set-valued framework, a well known obstacle for the development of adequate theoretical and numerical methods alike \cite{colonius2012dynamics}.

Recent developments in the theory of autonomous set-valued discrete dynamical systems ~\cite{kourliouros2023persistence,lamb2023attractors,lamb2025bifurcations,Tey2022minimal} have provided an effective approach in order to overcome this challenge, by introducing a finite dimensional and deterministic map, the so-called boundary map, defined on the unit (co)tangent bundle $T_1\mathbb{R}^d=\mathbb{R}^d\times \mathbb{S}^{d-1}$  of the phase space $\mathbb{R}^d$, which has the following characteristic property: smooth boundaries $\partial M$ of invariant sets $M$ of the set-valued map, lift in a unique way to outward unit (co)normal bundles $N_1^+\partial M$ which are invariant by the boundary map, and conversely. Using this deterministic map, the authors in the above references gained access to the differentiable properties of the boundaries of minimal invariant sets, and proved their smooth stability under small perturbations, indicating also several bifurcation scenaria in prototypical numerical examples.  

The present work establishes an analogous construction within the broader context of continuous-time set-valued dynamical systems arising from differential inclusion of the form \eqref{eq:diff-incl}. This general framework includes autonomous systems---where $f$ is time-independent, i.e.~$f(t,x)=f(x)$---as a special case. Instead of the boundary map of the discrete-time case, we obtain a boundary system of nonautonomous differential equations (see Proposition \ref{prop:boundary-trajectories-PMP}) on $T_1\mathbb{R}^d=\mathbb{R}^d\times \mathbb{S}^{d-1}$ of the form:   
\begin{equation}\label{eq:boundary-system-0}
     \begin{cases}
    \dot x=f(t,x)+\rho\,n,\\
    \dot n=-D_xf(t,x)^\top n+\langle n,D_xf(t,x)^\top n\rangle\, n.
\end{cases}
     \end{equation}  

The techniques for the analysis and derivation of the boundary system \eqref{eq:boundary-system-0}, and its correspondence to the set-valued system are more difficult than in the autonomous discrete-time case since in the discrete case the noise acts at the level of the trajectories whereas in continuous time it is a perturbation of the derivatives. 
Our results are therefore a nontrivial extension of the boundary map approach to the continuous time domain and are anchored to the Pontryagin Maximum Principle, a core result from optimal control theory (cf.~\cite{bressan2007introduction}, \cite{lee1986foundations}, \cite{pontryagin1962optimal}). Indeed, the experienced reader will immediately recognise that system \eqref{eq:boundary-system-0} is nothing but the normalisation on the unit sphere bundle of the Pontryagin--Hamilton system associated to the optimal control problem to minimise the distance  function to the boundary of the reachable set.  Łojasiewicz~\cite{lojasiewicz1980sufficiency} already uses the Pontryagin–Hamilton boundary system to parametrise and characterise the boundary of reachable sets on a finite interval whose length is fixed by a Riccati comparison. On that interval the reachable sets stay uniformly convex and the boundary remains continuously differentiable, so the costate is the outer normal and PMP extremals are exactly boundary trajectories. Beyond this interval the same argument does not, in general, extend, as the constants in \cite{lojasiewicz1980sufficiency} are sharp.   Our approach, instead, focuses on invariant sets, i.e.~time-dependent sets defined on the whole real line and we also consider the case in which the boundary is not necessarily smooth---a common scenario for nonlinear systems.

Using \eqref{eq:boundary-system-0} and a combination of control theoretic and nonautonomous dynamical systems techniques, we establish in Section \ref{sec:boundary-mapping-continuous} the first main result of the paper (Theorem \ref{thm:BM-NAODE-inv}), which concerns the correspondence principle between the boundary of an invariant set $\M$ of the differential inclusion \eqref{eq:diff-incl}, and invariant unit normal bundles (in fact invariant bundle of normalised Murdokhovich cones) $N_1^+\partial \M$ of the boundary system. This result proves and generalises a conjecture in \cite{Tey2022minimal}. Namely, we prove that the bundle of Murhokovich cones $N_1^+\partial \M$ is backward invariant with respect to the boundary system \eqref{eq:boundary-system-0} (which we will call the boundary flow), provided that $\M$ is invariant with respect to the set-valued evolution operator $\Phi^{t,t_0}$ of the differential inclusion \eqref{eq:diff-incl}. Moreover, in the case where the boundary $\partial \M$ of the invariant set $\M$ is the union of continuously differentiable ($C^1$-smooth) manifolds, we show that this correspondence elevates to the complete invariance (i.e.~both backward and forward) of the unit normal bundle $N_1^+\partial \M$ of the boundary, in accordance to the discrete case \cite{kourliouros2023persistence}. 

We strengthen our results in the context of nonautonomous linear differntial inclusions of the form \eqref{eq:diff-incl},  with $f(t,x)=L(t)x$, where $ L:\R\to\R^{d\times d}$ is a locally integrable function. Under the assumption of exponential stability of the linear differential equation $\dot x=L(t)x$, we show that a minimal attractor $\A$ for the set-valued system induced by the differential inclusion always exists (Theorem \ref{thm:exist-attractor}) and characterise many of its fiber-wise geometric and topological properties, such as, symmetry, strict convexity and closedness. Then, using the associated boundary system \eqref{eq:boundary-system-0}, we show the fiber-wise $C^1$-regularity of the boundary $\partial \A$ of the attractor (Theorem \ref{thm:linear-boundary-C1}). In equivalent terms, it means that the solution of the Hamilton-Jacobi-Bellman partial differential equation associated to the boundary system \eqref{eq:boundary-system-0}, is in fact a classical solution and not a viscosity one, an important property which, to our knowledge, has never been considered before in the existing literature. 
As a final result, we show that the outward unit normal bundle $N_1^+\partial \A$ of the boundary of the attractor $\A$ is a global (pullback and forward) attractor for the skew-product flow associated to the boundary system (Theorem \ref{thm:attractivity-UNB}), allowing a rigorous and reliable approximation of the boundary via numerical simulations as we show in an indicative example.


Before we close this section we would like to remark that there is a rich geometry underlying the boundary system \eqref{eq:boundary-system-0} which can be naturally interpreted in terms of the contact structure of the unit (co)tangent bundle $T_1\mathbb{R}^d$, as in the discrete case \cite{kourliouros2023persistence}. While it is not explicitly used in the results of the present paper, it is anticipated that further progress in the differentiable theory of minimal invariant sets (e.g. smooth stability, singularities and bifurcations) will greatly benefit, if not heavily rely, on the properties of the contact structure and the associated Legendrian manifold (and singularity) theory.  


\section{Notation and basic facts on set-valued dynamical systems}
In this brief section, we introduce the notation together with several fundamental notions and results concerning non-autonomous systems and the normal cone, in order to avoid disrupting the flow of the exposition later in the paper. Readers who wish to engage immediately with the principal questions addressed in this article may proceed directly to Section \ref{sec:boundary-mapping-continuous} and return to this section as necessary.

\subsection{Notation}
The usual notation for the sets of natural ($\N$), integer ($\Z$) and real numbers ($\R$) will be used. Analogously, $\R^d$, with $d\in\N$, will denote the $d$-dimensional Euclidean space with the usual norm $|x|$ for $x\in\R^d$ and scalar product $\langle x, y\rangle$ for $x,y\in\R^d$. 
Given a set $M\in\R^d$, $\overline M$ will denote the closure of $M$ with respect to the topology induced by the Euclidean norm.
The symbol $B_\rho(x)$ denotes the open ball in $\R^d$ of radius $\rho>0$ and centered at $x\in\R^d$. 
Moreover, for any set $M\subset\R^d$, the symbol $B_\rho(M)$ identifies the set $\{y\in\R^d\mid y\in B_\rho(x) \text{ for some } x\in M\}$, and for any pair of sets $U,V\subset \R^d$, the notation $U+V$ identifies the set $\{x\in\R^d\mid x=u+v,\text{ with } u\in U, v\in V\}$. Moreover, the unit sphere in $\R^d$ will be denoted by $\mathbb{S}^{d-1}$.
The set of matrices of dimension $m\times n$, with $n,m\in\N$, is denoted by  $\R^{m\times n}$ and given $A\in\R^{m\times n}$, $A^\top$ will denote its transpose. 
The space $\R^{m\times n}$ will be endowed with the induced norm defined by $\|A\|=\sup_{|x|=1}|Ax|$.

Oftentimes, we will deal the set $P_0(\R^d)$ of nonempty subsets of $\R^d$, and with the set $\K(\R^d)$  of nonempty compact subsets of $\R^d$. The latter will be endowed endowed with the Hausdorff distance $d_h:\K(\R^d)\times\K(\R^d)\to [0,\infty)$, defined by 
\[
d_h(M_1,M_2)= \max\{\dist(M_1,M_2),\dist(M_2,M_1)\},
\]
where $\dist(M_1,M_2)=\sup_{x\in M_1}\inf_{y\in M_2} |x-y|$. In particular, we note that $(\K(\R^d),d_h)$ is a metric space. For a nonempty set $M\subset\R^d$ we denote by $\partial M$ the topological boundary of $M$.

\subsection{Two-parameter semigroups and invariant nonautonomous sets}\label{subsec:nonautonomous}
In order to define a dynamical system in the presence of bounded noise or control, we firstly need to provide a description of the possible perturbations/controls for the system. For a bounded set $U\subset \R^d$ containing the origin, we consider the set of noise realisations 
\[
\U=\{\xi:\mathbb{R}\to \R^d\mid \xi(\mathbb{R})\subset U\}.
\]
Then, for each fixed noise realisation $\xi\in\U$ we further define the determinisitic nonautonomous two-parameter semigroup on $\R^d$ as,
\[
\Phi^{t,t_0}_\xi:\R^d\to \R^d,\quad x\mapsto\Phi^{t,t_0}_\xi(x),\quad \text{for all } t,t_0\in\mathbb{R} \text{ with } t>t_0.
\]
In particular, we require that, for every $t,t_0\in\mathbb{R}$ with $t>t_0$, every $x\in\R^d$ and every $\xi_0\in\U$,
\[
\sup_{s\in\mathbb{R}\cap[t,t_0]}|\Phi^{s,t_0}_\xi(x)-\Phi^{s,t_0}_{\xi_0}(x)|\to0\quad \text{as }  \sup_{s\in\mathbb{R}\cap[t,t_0]}|\xi(s)-\xi_0(s)|\to0.
\]
In other words, the evolution in time of every point $x\in\R^d$ is subjected to an evolution law which depends on time in two ways. The first is due to $\xi$, whereas the second is common to all the perturbations/controls. \smallskip

The collective effect of all the possible perturbations/controls can be studied via the two-parameter semigroup defined on the set of nonempty subsets  $P_0(\R^d)$ for all $t,t_0\in\mathbb{R} $ with $t>t_0$ via
\[
\Phi^{t,t_0}:P_0(\R^d)\to P_0(\R^d),\quad M\mapsto\Phi^{t,t_0}(M)=\bigcup_{x\in M, \xi\in\U}\Phi^{t,t_0}_\xi(x).
\]

A nonautonomous set is a family $\M=(M_t)_{t\in\R}\subset \R\times \R^d$. We refer to the subsets $M_t\subset \R^d$ for $t\in\R $ as to the fibres of $\M$.
Moreover, we say that  a nonautonomous set $\mathscr{M}=(M_t)_{t\in\R}\subset\R\times \R^d$ has a certain property, when such property is satisfied fibre-wise, e.g.~$\mathscr{M}$ nonempty if $M_t\neq\varnothing$ for all $t\in\R $, and compact if $M_t\subset\R^d$ is compact for all $t\in\R $. Note that this does not immediately imply that $\M$ is compact and not even bounded in $\R \times \R^d$. 
When the property identically holds for all the fibres of $\M$, then we say that it holds uniformly. For example, a nonautonomous set $\M$ is uniformly bounded if there is $\rho>0$ such that $M_t\subset B_\rho(0)$ for all $t\in\mathbb{R}$. 
A nonempty and compact  nonautonomous set $\M=(M_t)_{t\in\R}\subset\R\times \R^d$ is called \emph{forward invariant} for the two-parameter semigroup  $\Phi$ if 
\[
\Phi^{t,t_0}(M_{t_0} )\subseteq M_t,\quad\text{for all }t,t_0\in\R, \text{ with }t>t_0,
\]
and \emph{invariant} if the formula above holds with the equality. The existence of a forward invariant set for $\Phi$ is a sufficient condition for the existence of an invariant subset \cite{kloeden2011negatively}. 

Additionally, $\M$ is called a \emph{minimal invariant set} if it does not contain any proper non-empty closed subset which is invariant. Minimality is intended with respect to the compound effect of the perturbations/controls. Note that in the set-valued case the invariance is dynamically more relevant than in the single-valued case. In particular, minimal invariant sets are not individual trajectories in general. In this work, particular attention will be posed onto the boundary of nonautonomous invariant sets. For every $t\in\R$, we denote the boundary of the $t$-fibre $M_t$ of $\M$ with the symbol $\partial M_t$.

\subsection{Generalised normal cone and the outward unit normal bundle}\label{sec:generalized-unit-normal-bundle}
A central role in this work is played by the geometry of the boundary of an invariant set induced by certain differential inclusions. We hereby introduce classic notation that it is extensively used in our work. Given a compact set $M$ whose boundary $\partial M$ is a differentiable manifold,  the tangent space at a point $p\in \partial M$ is denoted by $T_p\partial M$, i.e.~the set of all tangent vectors at $p\in \partial M$. Moreover, we say that two vectors $v_1,v_2\in\R^d$ are normal, and write $v_1\bot v_2$, if the scalar product $\langle v_1, v_2\rangle=0$. The \emph{unit normal space} to $\partial M$ at $p$ is then denoted by $N_{1,p}\partial M$ and it consists of all the unit vectors which are normal to any element in $T_pM$. The \emph{outward unit normal bundle} is then defined as 
\begin{equation}\label{eq:normal-bundle}
N^+_1\partial M=\big\{(p,v)\in \partial M\times \mathbb{S}^{d-1}\mid v\in N_{1,p}\partial M \text{ outward-pointing}\big\},
\end{equation}
where the term \emph{outward-pointing} means that $p+\ep v\notin M$ for $\ep>0$ sufficiently small.

In some cases, we deal with topological manifolds which admit points of non-differentiability and the above definition for the outward unit normal bundle is not applicable. In order to deal with such cases it is necessary to consider a generalised normal cone. Let \( M \subset \mathbb{R}^d \) be compact and such that its boundary $\partial M$ is a (possibly non-differentiable) manifold. The \emph{Fréchet unit normal cone} at \( p\in \partial M \), is given by
\[
\widehat{UN}_p\partial M := \left\{ v \in \mathbb{S}^{d-1} \;\middle|\; \limsup_{y \to p,\, y \in \partial M} \frac{\langle v, y - p \rangle}{|y - p|} \leq 0 \right\}.
\]
The Fréchet \emph{outward} unit normal cone $\widehat {UN}^+_p\partial M$ is the set  of the outward-pointing vectors in the Fréchet normal cone. The \emph{Mordukhovich  outward unit normal cone} (also called the \emph{limiting  outward unit normal cone}) to \( \partial M \) at \( p \), is defined as
\[
N^+_{1,p}\partial M := \left\{ v \in \mathbb{S}^{d-1} \;\middle|\; \exists\, p_k \to p,\ v_k \to v,\ p_k \in \partial M,\ v_k \in \widehat{N}^+_{1,p_k}\partial M \right\}.
\]
When $\partial M$ is a smooth manifold, the Mordukhovich outward unit normal cone at a point coincides with the classical outward unit  normal space at that point, which justifies the apparent abuse of notation.
The outward unit normal bundle can therefore be defined as in \eqref{eq:normal-bundle}. 

\section{The boundary system for nonautonomous differential inclusions}\label{sec:boundary-mapping-continuous}
Our work deals with 
a family of nonautonomous differential equations of the type
\begin{equation}\label{eq:1}
    \dot x = f(t,x) + \rho\xi(t),\quad t\in\R,\, x\in\R^d,
\end{equation}
where $f:\R \times \R^d\to\R^d$ is a function smooth enough to guarantee existence and uniqueness of solutions for $\dot x=f(t,x)$---as we shall better clarify in the following---and $\xi:\R\to\R^d$ is a function in the set $\U=\{\xi\in L^\infty(\R,\R^d)\mid \|\xi\|_\infty\le1\}$ representing all the possible bounded perturbations/controls with values essentially in the unit sphere of $\R^d$. A unique solution to \eqref{eq:1} for each initial value and forcing $\xi$ exists in the sense of Carathéodory, i.e.~a locally absolutely continuous function satisfying the integral problem,
\[
x(t,t_0,x_0,\xi)=x_0+\int_{t_0}^t\Big(f\big(s,x(s)\big)+\rho\xi_s\Big)\,ds.
\]
In view of the above, we require that $f$ satisfies the weakest sufficient conditions for the existence and uniqueness of solutions, i.e.~the Carath\'eodory conditions,
\begin{itemize}[leftmargin=*]
    \item (local integrability) for any compact set $K\subset\R^d$ there exists a locally integrable function $m\in L^1_{loc}(\R)$ such that for almost every $t\in\R$, \[|f(t,x)|\le m(t) \quad \text{for all }x\in K;\]
    \item (local Lipschitz continuity) for any compact set $K\subset\R^d$ there exists a locally integrable function $l\in L^1_{loc}(\R)$ such that for almost every $t\in\R$, 
    \[|f(t,x)-f(t,y)|\le l(t)|x-y| \quad \text{for all }x,y\in K;\]
\end{itemize}
We will need a bit more regularity in the variable $x$ in order to study stability and parallel displacement of tangent spaces. Hence, we also assume that 
\begin{itemize}[leftmargin=*]
    \item (continuous differentiability) $f$ is continuously differentiable in the variable $x$ for almost every $t\in\R$, and we  denote by $D_xf$ the Jacobian of $f$ with respect to the $x$-variables. Moreover, we assume that for any compact set $K\subset\R^d$ and up to suitably rescaling $m\in L^1_{loc}(\R)$ we have that for almost every $t\in\R$, \[|D_xf(t,x)|\le m(t) \quad \text{for all }x\in K;\]
\end{itemize}
Since our interest lies in the study of invariant sets (cf.~Section \ref{subsec:nonautonomous}) and in particular attractors, we also assume that the solutions to~\eqref{eq:1} are globally defined (a priori bounds assuring this can be found in \cite{hale2009ordinary, bressan2007introduction}). 
A theorem by Filippov guarantees that an absolutely continuous map $x: [a,b]\to \R^d$ is a trajectory of~\eqref{eq:1} if and only if it satisfies the differential inclusion
\begin{equation}
    \label{eq:set-valued-systems}
\dot x\in F_\rho(t,x),\quad\text{where } F_\rho(t,x)=\overline B_\rho\big(f(t,x)\big),
\end{equation}
almost everywhere (see Theorem 3.1.1 in \cite{bressan2007introduction} for the proof in an even more general setting). Moreover, for any fixed $t_0<t$, the set of trajectories of the differential inclusion above is closed in $C([t_0,t],\R^d)$ \cite[Theorem 3.3.1]{bressan2007introduction}. 
We shall denote by $\Phi^{t,t_0}$ the evolution operator defined by,
\[
\Phi^{t,t_0}(x_0)=\{y\in\R^d\mid y=x(t,t_0,x_0,\xi) \text{ for } \xi\in\U\},
\] 
and neglect the dependence on $\rho$ from the notation (as $\rho$ remains fixed).
 Under the above mild assumptions, the map $x\mapsto \Phi^{t,t_0}(x)$ is upper semicontinuous, 
 it has compact values (\cite[Corollary 7.1]{deimling2011multivalued} and \cite{aubin1984differential}), but it is also lower semicontinuous (and thus continuous) due to the Lipschitz continuity of $f$ in $x$ \cite[Corollary 2.4.1]{aubin1984differential}. 
 This implies in particular that the image of a nonempty compact set $K\subset \R^d$ through $\Phi^{t,t_0}$ is also compact. Indeed, considering a sequence $(y_n)_\nin$ of elements in $\Phi^{t,t_0}(K)$, there is a sequence $(x_n)_\nin$ of elements in $K$ such that $y_n\in \Phi^{t,t_0}(x_n)$ for all $\nin$ by definition. Since $K$ is compact, then $(x_n)_\nin$ converges, up to a subsequence, to an element $\overline x\in K$. Hence, up to a subsequence, and using the continuity of $x\mapsto \Phi^{t,t_0}(x)$, we have that the following limit inclusion holds,
 \[
 \lim_\nti y_n\in  \lim_\nti \Phi^{t,t_0}(x_n)=\Phi^{t,t_0}(\overline x)\subset \Phi^{t,t_0}(K).
 \]
 Therefore, we obtain the following two-parameter semigroup acting on the set $\K(\R^d)$ of nonempty compact subsets of $\R^d$,
\[
\Phi^{t,t_0}:\K(\R^d)\to \K(\R^d),\quad M\mapsto\Phi^{t,t_0}(M)=\bigcup_{x\in M}\Phi^{t,t_0}(x),\quad \text{for all } t>t_0. 
\]
Note that, given $M\in \K(\R^d)$, the map  $t\mapsto \Phi^{t,t_0}(M)$ satisfies,
\[
d_h(M,\Phi^{t,t_0}(M))= d_h\big(\Phi^{t_0,t_0}(M ),\Phi^{t,t_0}(M_{t_0} )\big)\xrightarrow{t\to t_0} 0.
\]

\par\smallskip

Our first result shows that trajectories that remain on the boundary of an invariant set for a differential inclusion induced by \eqref{eq:1}, satisfy an optimal control problem with respect to the distance function to the boundary of the invariant set. This is a consequence of the Pontryagin's Maximum Principle, a central result in optimal control theory as we show next.

\begin{prop}\label{prop:boundary-trajectories-PMP}
     Consider an invariant set $\M=(M_t)_{t\in\R}$ for the set-valued mapping $\Phi$ with respect to $\U$. For any $[\tau_1,\tau_2]\subset\R$, if $\tau_1<t_0<\tau_2$, $x_0\in\partial M_{t_0}$ and $\xi^*\in L^\infty([\tau_1,\tau_2])$ with $\|\xi\|_\infty\le 1$ is such that $x^*(t):=x(t,t_0,x_0,\xi^*)\in\partial M_t$ for all $t\in[\tau_1,\tau_2]$, then $\xi^*$  minimises the cost function 
     \begin{equation}\label{eq:cost-function}
C(\xi,\tau_1,\tau_2)=\int_{\tau_1}^{\tau_2} f_0(t,x(t,\xi))\,dt\quad\text{where } f_0(t,x):=\inf_{p\in\partial M_t}|x-p|^2.
\end{equation}
Moreover, $\big(x^*(t),\xi^*(t)\big)$ solves the system of differential equations
\begin{equation}\label{eq:boundary-system}
     \begin{cases}
    \dot x=f(t,x)+\rho\,u,\\
    \dot u=-D_xf(t,x)^\top u+\langle u,D_xf(t,x)^\top u\rangle\, u.
\end{cases}
     \end{equation}
\end{prop}
\begin{proof}

Let $t_0\in[\tau_1,\tau_2]$, $x_0\in\partial M_{t_0}$ and assume that the solution $x(t,t_0,x_0,\xi^*)$ to~\eqref{eq:1} with $\xi=\xi^*$ and $x(t_0)=x_0$ satisfies $x(t,t_0,x_0,\xi^*)\in\partial M_t$ for all $t\in[\tau_1,\tau_2]$. Note that for every $t\in\R$ the zero-level set for $f_0(t,x)$ coincides with the points in $\partial M_t$ and the cost function \eqref{eq:cost-function} penalises the responses $x(t,t_0,x_0,\xi)$ whose graphs do not lie on  $(\partial M_t)_{t\in\R}$. 
Since the cost function \eqref{eq:cost-function} is non-negative and $C(\xi^*,\tau_1,\tau_2)=0$ by construction, then it is immediate that $\xi^*\in L^\infty([\tau_1,\tau_2])$ minimises \eqref{eq:cost-function}. \smallskip

In order to show that $\big(x^*(t),\xi^*(t)\big)$ solves \eqref{eq:boundary-system}, we proceed in two steps. Firstly, we construct an Hamiltonian system which simultaneously takes care of \eqref{eq:1} and \eqref{eq:cost-function} and use the fact that $\xi^*$, as a minimiser of \eqref{eq:cost-function}, needs to satisfy the necessary condition for optimality known as Pontryagin Maximum Principle. Secondly, we normalise the action in the  Hamiltonian system to obtain \eqref{eq:boundary-system}.\smallskip

\emph{Step 1: $u^*$ satisfies Pontryagin's Maximum Principle.}
Let
\[
\widetilde x(t)=\left(x^0(t),x(t),t\right)^\top,\quad t\in[\tau_1,\tau_2],
\]
 be the unique solution of the extended initial value problem, 
\begin{equation}\label{eq:time-augmented-response}
\dot{\widetilde x}=\widetilde f(\widetilde x,u)
\quad\leftrightarrow\quad
\begin{cases}
    \dot x^0=f_0(t,x)\\
    \dot x= f(t,x)+\rho u(t)\\
    \dot t=1
\end{cases}
\quad\widetilde x(t_0)=
\begin{pmatrix}
0\\ x_0\\t_0    
\end{pmatrix}.
\end{equation}
Note that $x(t)$ coincides with $x^*(t)=x(t,t_0,x_0,\xi^*)$ if $u=\xi^*$, and in this case we will denote the extended solution by $\widetilde x^*$. Let us consider the Hamiltonian function, 
\begin{equation}\label{eq:hamiltonian-function}
\widetilde H(\widetilde \eta,\widetilde x,u)=\eta_0 f_0(t,x)+\eta_1 (f_1(t,x)+\rho u_1)+\dots+\eta_d(f_d(t,x)+ \rho u_d)+\eta_{d+1}
\end{equation}
which produces the Hamiltonian system
\[
\begin{cases}\displaystyle
    \dot{\widetilde x}_j= \frac{\partial\widetilde H}{\partial\widetilde\eta_j}(\widetilde \eta,\widetilde x,u)\quad j=0,\dots,d+1,\\[2ex] \displaystyle
    \dot{\widetilde\eta}_j= -\frac{\partial\widetilde H}{\partial\widetilde x_j}(\widetilde \eta,\widetilde x,u)\quad j=0,\dots,d+1,
\end{cases}
\]
where the first $d+2$ equations are exactly those in \eqref{eq:time-augmented-response}, while the last $d+2$ equations define the time-augmented adjoint system, 
\begin{equation}\label{eq:time-augmented-adjoint}
\dot{\widetilde\eta}=-\widetilde\eta\, \frac{\partial \widetilde f}{\partial \widetilde x}\big(\widetilde x(t)\big)\quad\leftrightarrow\quad
\begin{cases}
    \dot \eta^0=0\\
    \dot \eta_j= -\sum_{i=0}^d\eta_i \frac{\partial f_i}{\partial x_j}\big(t,x(t)\big),\quad j=1,\dots,d\\
    \dot \eta_{d+1}=-\sum_{i=0}^d\eta_i \frac{\partial f_i}{\partial t}\big(t,x(t)\big).
\end{cases}
\end{equation}
Since $C(\xi^*,\tau_1,\tau_2)=0$, i.e.~$\xi^*$ is optimal, then it must satisfy Pontryagin's Maximum Principle (see for example Theorem 2 in \cite{clarke2010pontryagin}), i.e.~there must exist a nontrivial time-augmented adjoint response $\widetilde \eta^*(t)$ (solving \eqref{eq:time-augmented-adjoint}) such that
\begin{equation}\label{eq:PMP}
\widetilde H(\widetilde \eta^*(t),\widetilde x^*(t),\xi^*(t))=\max_{u\in\U}\widetilde H(\widetilde \eta^*(t),\widetilde x^*(t),u(t))\quad \text{for a.a.~} t\in [\tau_1,\tau_2]. 
\end{equation}  

\emph{Step 2: $\big(x^*(t),\xi^*(t)\big)$ solves \eqref{eq:boundary-system}.}
Note that for every $t\in\R$ fixed, $x\mapsto f_0(t,x)$ is a differentiable function (the partial derivatives exist at every point and are continuous), and its zero-level set---the function minimum---corresponds to points in $\partial M_t$. Hence, and because of $|x-p|^2$, we have that the necessary condition for criticality,
\begin{equation}\label{eq:partial_f0=0}
D_xf_0(t,x)=0,\quad \text{for all } x\in\partial M_t,
\end{equation}
must be satisfied. 

In particular, given the formula \eqref{eq:hamiltonian-function}, it must be that $(\eta^*(t))^\top$ and $\xi^*(t)$ are parallel for all $t\in[[\tau_1,\tau_2]$ and furthermore $\|\xi^*\|_\infty=1$. Then, we can write $\xi^*(t)=\tfrac{(\eta^*(t))^\top}{|\eta^*(t)|}$, and  we have that 
\[
\begin{split}
\frac{d \xi^*}{dt}&=\frac{(\dot \eta^*)^\top|\eta^*|-(\eta^*)^\top\frac{d}{dt}|\eta^*|}{|\eta^*|^2}
=\frac{(\dot \eta^*)^\top}{|\eta^*|}-\frac{2(\eta^*)^\top\,\eta^*(\dot \eta^*)^\top}{2|\eta^*|^3}\\[1ex]
&=\frac{1}{|\eta^*|}\left(-\eta_0D_x f_0(t,x)-D_xf(t,x)^\top \eta^*\right)+\\
&\qquad\qquad\qquad\qquad\quad\ \, +\frac{(\eta^*)^\top\,\langle \eta^*,\eta_0D_x f_0(t,x)+D_xf(t,x)^\top \eta^*\rangle}{|\eta^*|^3}\\
&=-\frac{\eta_0}{|\eta^*|}D_x f_0(t,x)-D_xf(t,x)^\top \xi^*+\\
&\qquad\qquad\qquad\qquad\quad\ \, +\left(\langle  \xi^*,D_xf(t,x)^\top u^*\rangle +\left\langle  \xi^*,\frac{\eta_0}{|\eta^*|} D_x f_0(t,x)\right\rangle\right)  \xi^*.
\end{split}
\]
Hence, due to \eqref{eq:partial_f0=0}, we have that $\big(x^*(t),\xi^*(t)\big)$ must satisfy \eqref{eq:boundary-system}. 
\end{proof}

\begin{rmk}
    The proof of Proposition \ref{prop:boundary-trajectories-PMP} reveals  that the boundary system \eqref{eq:boundary-system} is in fact the Pontryagin system up to constraining the solution of the equation for the normal $n$ to the unit sphere via normalisation.
\end{rmk}

Next, we prove a basic property guaranteeing that trajectories starting in the interior of an invariant set for \eqref{eq:diff-incl} (cf.~Section \ref{subsec:nonautonomous}), cannot reach the boundary in finite time.

\begin{prop}\label{prop:trajectories-on-boundary}
    Let $\M=(M_t)_{t\in\R}$ be an invariant set for the set-valued semi-flow $\Phi^{t,t_0}$ with respect to $\U$. Then, for any $t>t_0$, any $x_0\in M_{t_0}\setminus\partial M_{t_0}$, and any function $\xi\in L^\infty(\R,\R^d)$, with
    $\|\xi\|_\infty\le 1$, the solution $x(\cdot,t_0,x_0,\xi)$ to $\dot x=f(\tau,x)+\rho \xi(\tau)$ with with $x(t_0)=x_0$ satisfies $x(t,t_0,x_0,\xi)\notin\partial M_t$. 
\end{prop}
\begin{proof}
Assume, by contradiction, that there are $t>t_0$, $x_0\in M_{t_0}\setminus \partial M_{t_0}$ and a function $\xi:\R\to \overline B_1(0)$ such that $x(t,t_0,x_0,\xi)\in\partial M_t$. Then, a constant $\ep=\ep(t_0,x_0)>0$ must exists such that $B_\ep(x_0)$ is also contained in the interior of $M_{t_0}$. 
However, since $x(t,t_0,x_0,\xi)\in\partial M_t$, then, by continuity and uniqueness of solutions, the set of points $x(t,t_0,B_\ep(x_0),\xi)$ has non-void intersection with the complement of $\overline M_t$ in $\R^d$, which is not possible due to the invariance of $\M$. Therefore, it must be that such $t>t_0$, $x_0\in M_{t_0}\setminus \partial M_{t_0}$ and $\xi:\R\to \overline B_1(0)$ do not exist.
\end{proof}

\begin{rmk}\label{rmk:backward-time-PMP}
As a consequence of the previous two results, we have the following important property about the past of points on the boundary of an invariant set. If $t_0\in\R$ and $x_0\in\partial M_{t_0}$, then, a function $\xi\in L^\infty(\R,\R^d)$, with   $\|\xi\|_\infty\le 1$ exists, such that the solution $x(\cdot,t_0,x_0,\xi)$ to $\dot x=f(\tau,x)+\rho \xi(\tau)$ with $x(t_0)=x_0$ satisfies $x(t,t_0,x_0,\xi)\in\partial M_t$ for all $t\le t_0$. Moreover, $(x(t,t_0,x_0,\xi),\xi(t))$ solves the differential system \eqref{eq:boundary-system}.  Indeed, since no point on the boundary can be reached in finite time by a trajectory inside the invariant set, and due to the invariance itself, such a function $\xi$ must exist and the solution $x(t,t_0,x_0,\xi)\in\partial M_t$ for all $t\le t_0$. Then, Proposition \ref{prop:boundary-trajectories-PMP} applies.
\end{rmk}

Next, we show that the conclusions of Remark \ref{rmk:backward-time-PMP} can be strengthened to the backward invariance of the outward unit normal bundle of the invariant set $\M=(M_t)_{t\in\R}$ with respect to the system \eqref{eq:boundary-system}. Recall that when the boundary $\partial M$ of $\M$ is a possibly non-differentiable manifold, the outward unit normal bundle at every fiber $\partial M_t$ is defined by 
\[
N^+_1\partial M_t=\big\{(p,v)\in \partial M_t\times \mathbb{S}^{d-1}\mid v\in  \partial N_{1,p}\partial M_t \text{ outward-pointing}\big\},
\]
where $N_{1,p}\partial M_t$ is the Mordukhovich outward unit normal cone at a point $p\in\partial M_t$ (cf.~Section \ref{sec:generalized-unit-normal-bundle}), and $\partial N_{1,p}\partial M_t$ is its boundary.
We call the \emph{(Mordukhovich) outward unit normal bundle} the nonautonomous set 
\[
N^+_1\partial \M=(N^+_1\partial M_t)_{t\in\R}.
\]

\begin{prop}\label{prop:backward-invariance}
  Consider an invariant set $\M=(M_t)_{t\in\R}$ for the set-valued semi-flow $\Phi^{t,t_0}$ with respect to $\U$. Then, the outward unit normal bundle $N^+_1\partial \M$ is backward invariant under the two-parameter semigroup induced by \eqref{eq:boundary-system}.
 \end{prop}
 \begin{proof}
 Consider the solution $(x^*(t),\eta(t))$ of 
\[
     \begin{cases}
    \dot x=f(t,x)+\rho\,\eta/|\eta|,\\
    \dot \eta=-D_xf(t,x)^\top \eta.
\end{cases}
\]
with initial condition $(\overline x, \overline n)\in N^+_1 M_t$ at time $t$. 
Moreover, let $n^*(\tau):=\eta(\tau)/|\eta(\tau)|$ and notice that, by construction, $\|n^*\|_\infty=1$. Then, the solution $x(\tau,t, \overline x, n^*)$ to 
\[
\dot x= f(t,x)+\rho n^*(t), \quad x(t)=\overline x
\]
satisfies  $x(\tau,t, \overline x, n^*)\in M_\tau$ for all $\tau\in\R$ because $\M$ is invariant by assumption, and since $\overline x\in\partial M_t$, it must also be that $x(\tau,t, \overline x, n^*)\in \partial M_\tau$ for all $\tau<t$ thanks to Proposition \ref{prop:trajectories-on-boundary}. Hence, Proposition \ref{prop:boundary-trajectories-PMP} guarantees that $\big(x(\tau,t, \overline x, n^*),n^*(\tau)\big)$ is a solution to  \eqref{eq:boundary-system}. That is,
\[
x(\tau,t, \overline x, n^*)=x^*(\tau)\quad\text{for all }\tau\in\R.
\]
It is left to prove that $n^*(\tau)\in \partial N^+_{1,x^*(\tau)}\partial M(\tau)$ for almost all $\tau < t$.  
Let $H_{t}$ be the hyperplane intersecting $\partial M_t$ at $\overline x$ which is normal to $\overline n$. Note that $H_{t}$  undergoes a parallel displacement to $H_{\tau}$ via the variational equation along $x^*(\tau)$ i.e.~$H_\tau=\Psi(\tau,t)H_t$.
Hence, taken any $v_0\in H_t$, the unique solution $v(\tau)$ to the initial value problem 
     \[
     \dot v =D_xf(\tau,x^*(\tau))v,\quad v(t)=v_0,
     \]
     and the solution to the adjoint problem
    \[
    \dot \eta =-\eta D_xf(\tau,x^*(\tau)),\qquad \eta(t)=\overline n^\top,
    \]
    we have that,
    \[    \eta(\tau)v(\tau)=\overline n^\top\Psi(\tau,t)^{-1}\Psi(\tau,t)v_0=\overline n^\top v_0=0\quad\text{for all } \tau\in[t_0,t].
    \]
    which means that $n^*(\tau)=\eta(\tau)^\top/|\eta(\tau)|\bot H_\tau$ for all $\tau\in[t_0,t]$, hence $\big(x^*(\tau),n^*(\tau)\big)\in N^+_{1}\partial M_{\tau}$ for almost all $\tau < t$.
 \end{proof}

Having discussed the past of points on the boundary of an invariant set for the differential inclusion \eqref{eq:set-valued-systems}, we now focus on their future. 
Under the additional assumption of differentiability of the boundary of the invariant set $\M$, we can show that points on the boundary admit at least one trajectory that also remains on the boundary for future times.
\begin{prop}\label{prop:boundary-forward-invariant}
    Consider a nonautonomous set $\M=(M_t)_{t\in\R}$ that is invariant for the set-valued two-parameter semigroup $\Phi^{t,t_0}$. If $\M$ has (fibre-wise) continuously differentiable boundary, then for any $t_0\in\R$ and any $x_0\in \partial M_{t_0}$  there is $\xi\in L^\infty$ with $\|\xi\|_\infty\le1$ such that $x(t,t_0,x_0,\xi)\in\partial M_t$  for all $t>t_0$.
\end{prop}
\begin{proof}
Fix $t_0\in\R$, and let $G_{t_0}$ be the set defined by 
\[
G_{t_0}:=\{x_0\in\partial M_{t_0}\mid x(t,t_0,x_0,\xi)\in\partial M_t\text{ for all }t>t_0, \text{ for some }
\xi\in L^\infty,\, \|\xi\|_\infty\le1\}.
\]
Note that this set is closed (and therefore compact) because its complementary set $B_{t_0}=\partial M_{t_0}\setminus G_{t_0}$ is open. We aim to show that $B_{t_0}=\varnothing$. Fix $t>t_0$ and consider the function 
\[
F:G_{t_0}\to\partial M_t, \quad x_0\mapsto x(t,t_0,x_0,\xi)
\]
Note that $F$ is surjective because $\M$ is invariant by assumption, and it is not possible to reach the boundary from the interior in finite time thanks to Proposition~\ref{prop:trajectories-on-boundary}.
Moreover, $F$ is  injective. Indeed,  let $x_1,x_2\in G_{t_0}$ such that $x(t,t_0,x_1,\xi_1)=x(t,t_0,x_2,\xi_2)=:p\in \partial M_t$. However, since $\M$ has (fibre-wise) continuously differentiable boundary by assumption, there is a unique $n\in\mathbb{S}^{d-1}$ such that $n\bot T_p\partial M_t$ outward directed. Then, the backward invariance of the outward unit normal bundle proven in Theorem \ref{prop:backward-invariance} and the uniqueness of solutions for the Cauchy problem \eqref{eq:boundary-system} with initial value $(p,n)$ at time $t$ guarantee that $x(s,t_0,x_1,\xi_1)=x(s,t_0,x_2,\xi_2)$ for all $s\in[t_0,t]$ and thus $x_1=x_2$. Additionally, $F$ is continuous, due to continuous variation of solutions. Hence, $F$ is a continuous bijection from a compact space into an Hausdorff space and therefore it is also an homeomorphism. Therefore, $B_{t_0}$ must be empty because otherwise, $F$ would be an homeomorphism between a manifold with boundary and one without, which is a contradiction.
\end{proof}

We are therefore ready to state and prove the following  result  on the invariance of the outaward unit normal bundle of an invariant set with respect to the two-parameter semigroup induced by the boundary system \eqref{eq:boundary-system}, as it was originally conjectured in~\cite{Tey2022minimal}.

 \begin{thm}\label{thm:BM-NAODE-inv}
     Consider an invariant set $\M=(M_t)_{t\in\R}$ for the set-valued semigroup $\Phi^{t,t_0}$ with respect to $\U$. Then, the following statements hold true.
     \begin{itemize}[leftmargin=18pt]
         \item[(i)] The outward unit normal bundle $N^+_1\partial \M$ is backward invariant under the two-parameter semigroup induced by the boundary system: 
         \begin{equation}\label{eq:boundary-flow}
     \begin{cases}
    \dot x=f(t,x)+\rho\,n,\\
    \dot n=-D_xf(t,x)^\top n+\langle n,D_xf(t,x)^\top n\rangle\, n.
\end{cases}
     \end{equation}
     \item[(ii)] If additionally $\M$ has (fibre-wise) continuously differentiable boundary, then the outward unit normal bundle $N^+_1\partial \M$ is invariant under the two-parameter semigroup induced by the boundary system \eqref{eq:boundary-flow}.
     \end{itemize}
 \end{thm}

\begin{proof} 
The backward invariance is shown in Proposition \ref{prop:backward-invariance}. If $\M$ has (fibre-wise) continuously differentiable boundary, then Propositions \ref{prop:boundary-forward-invariant} and \ref{prop:boundary-trajectories-PMP} together with the argument for the parallel displacement of the tangent space presented in  Proposition \ref{prop:backward-invariance}, guarantee that the solution $\big(x^*(\tau),n^*(\tau)\big)$ to \eqref{eq:boundary-flow} with initial condition  $(\overline x, \overline n)\in N^+_1 M_{t_0}$ satisfies $\big(x^*(\tau),n^*(\tau)\big)\in  N^+_{1}\partial M_{\tau}$ for almost all $\tau > t_0$, which concludes the proof.
\end{proof}

\section{Minimal attractors of nonautonomous linear differential inclusions}
In this section, we turn our attention to the case of nonautonomous differential inclusions induced by linear systems. We show that, under the assumption of exponential stability, there is a (unique) uniformly bounded minimal invariant set for the set-valued system which is both forward and pullback attracting, and we study some of its geometric and topological properties including simmetry, convexity, closedness and smoothness of the boundary. Then, we consider the boundary system for such minimal invariant set. Hence, we introduce the skew-product formalism and show that the outward unit normal bundle is in fact the pullback attractor for the skew-product flow induced by the boundary system. The pullback attractor induces a global attractor for the skew-product flow under the classic assumption of a compact flow on the base. Finally, we apply this results to a concrete example showing how the deterministic nature of the boundary system, together with the property of attractivity of the outward unit normal bundle allow to numerically approximate the boundary of the attractor of the set-valued system by simply carrying out a Runge-Kutta-type integration.

\subsection{General properties and existence of an attractor}
Consider the nonautonomous linear ordinary differential equations, 
\begin{equation}\label{eq:linearNAODEwNOISE}
    \dot x =L(t)x+ \rho \xi(t)\quad t\in\R,\, x\in\R^d,\, \xi_t \in \overline B_1(0),
\end{equation}
where $ L:\R\to\R^{d\times d}$ is a locally integrable function and the functions $\xi:\R\to\R^d$ belong to the space $\U=\{\xi\in L^\infty(\R,\R^d)\mid \|\xi\|_\infty\le1\}$.  Our first observation attains the continuity of the map $x\mapsto \Phi^{t,t_0}(x)$.
\begin{lem}
    The set-valued map $x\mapsto \Phi^{t,t_0}(x)$ induced by \eqref{eq:linearNAODEwNOISE} is continuous.
\end{lem}
\begin{proof}
    Firstly, note that $x\mapsto \Phi^{t,t_0}(x)$ is upper semi-continuous 
(\cite[Corollary~7.1]{deimling2011multivalued}
 and \cite{aubin1984differential}). In order to prove the lower semi-continuity, consider any generalised sequence $(x_\mu)$ converging to some $x_0$ and any $y_0\in \Phi^{t,t_0}(x_0)$. Then, there exists $\xi_0\in\U$ such that 
\[
y_0=\Psi(t,t_0)x_0+\int_{t_0}^t\Psi(t,s)\xi(s)\,ds.
\]
Therefore, the generalised sequence defined by
\[
y_\mu=\Psi(t,t_0)x_\mu+\int_{t_0}^t\Psi(t,s)\xi(s)\,ds,
\]
is such that $y_\mu\in \Phi^{t,t_0}(x_\mu)$ and additionally  $(y_\mu)$ converges to  $y_0$. As a consequence, the map $x\mapsto \Phi^{t,t_0}(x)$ is also lower semicontinuous and therefore continuous.
\end{proof}

In order to have a dynamically relevant minimal invariant set for the set-valued dynamical system associated to \eqref{eq:linearNAODEwNOISE}, we are going to assume that the linear homogeneous system \begin{equation}\label{eq:linearNAODE}
\dot x=L(t)x,    
\end{equation} 
is uniformly asymptotically stable. This is equivalent to having exponential stability of the trivial solution (see for example \cite[Theorem~4.4.2]{adrianova1995introduction}), i.e.~there are constants $K\ge1$ and $\gamma>0$ such that,
\[
\|\Psi(t,s)\|\le Ke^{-\gamma(t-s)},\quad\text{for all }t\ge s,
\]
where $\Psi(t,s)$ is the principal matrix solution of \eqref{eq:linearNAODE} at $s\in\R$. We shall say that \eqref{eq:linearNAODE} is exponentially stable.

\begin{thm}\label{thm:exist-attractor}
    Assume that~\eqref{eq:linearNAODE} is exponentially stable. Then the set $\A=( A_t)_{t\in\R}$ defined by    \begin{equation}\label{eq:hyp_sols_lin}
         A_t=\left\{x\in\R^d\mid x=\rho\int_{-\infty}^t\Psi(t,s)\xi(s)\,ds\text{ for }\xi\in\U\right\},
    \end{equation}
    is uniformly bounded and it is the unique minimal invariant set for the two-parameter semigroup $\Phi$ induced by the differential inclusion $\dot x\in \overline B_\rho(L(t)x)$. Moreover, for every compact set $ K\subset\R^d$ we have that 
    \[
    \lim_{t_0\to-\infty} h(\Phi^{t,t_0}(K), A_t)=0,\qquad \text{for all } t\in\R, 
    \]
    and 
    \[
    \lim_{t\to\infty} h(\Phi^{t,t_0}(K), A_t)=0,\qquad \text{for all } t_0\in\R, 
    \]
    that is, $ \A$ pullback and forward attracts every compact subset of $\R^d$.
\end{thm}
\begin{proof}
     For any fixed $\xi\in\U$, the non-homogeneous linear problem \eqref{eq:linearNAODEwNOISE} has a unique uniformly attracting (hyperbolic) solution defined by 
    \[
    x_\xi(t)=\rho\int_{-\infty}^t\Psi(t,s)\xi(s)\,ds,
    \]
    obtained as the pullback limit of any solution to \eqref{eq:linearNAODEwNOISE} written through the variation of constants formula. We shall firstly prove the invariance of $\A$, i.e.~the equality $\Phi^{t,t_0}(A_{t_0})= A_t$ for all $t_0<t$.

    $(\subset)$ Consider $a\in\Phi^{t,t_0}(A_{t_0})$. Then, there is $\zeta\in\U$ such that 
    \[
    \begin{split}
    a&\in\Phi^{t,t_0}\bigg(\rho\int_{-\infty}^{t_0}\Psi(t_0,s)\zeta(s)\,ds\bigg)\\
    &=\bigcup_{\xi\in\U}\Big[\Psi(t,t_0)\rho\int_{-\infty}^{t_0}\Psi(t_0,s)\zeta(s)\,ds+\rho\int_{t_0}^t\Psi(t,s)\xi(s)\,ds\Big]\\
    &=\bigcup_{\xi\in\U}\Big[\rho\int_{-\infty}^{t_0}\Psi(t,s)\zeta(s)\,ds+\rho\int_{t_0}^t\Psi(t,s)\xi(s)\,ds\Big]\\
    &=\bigcup_{\xi\in\U}\Big[\rho\int_{-\infty}^{t}\Psi(t,s)\big(\zeta(s)\mathds{1}_{(-\infty,t_0)}(s)+\xi(s)\mathds{1}_{(t_0,t)}(s)\big)\,ds\Big],
     \end{split}
    \]
    where $\mathds{1}_I(s)$ is the indicator function for $I\subset\R$. Note that $\zeta(s)\mathds{1}_{(-\infty,t_0)}(s)+\xi(s)\mathds{1}_{(t_0,t)}(s)\in\U$ and therefore the right-hand side of the previous chain of equalities is contained in $ A_t$.

    $(\supset)$ Let $t_0,t\in\R$ with $t_0<t$ and $a\in  A_t$. Then, there is $\zeta\in\U$ such that 
    \[
    \begin{split}
    a= x_\zeta(t)&=\rho\int_{-\infty}^{t}\Psi(t,s)\zeta(s)\,ds=\rho\int_{-\infty}^{t_0}\Psi(t,s)\zeta(s)\,ds+\rho\int_{t_0}^{t}\Psi(t,s)\zeta(s)\,ds\\
    &=\Psi(t,t_0)\rho\int_{-\infty}^{t_0}\Psi(t_0,s)\zeta(s)\,ds+\rho\int_{t_0}^t\Psi(t,s)\zeta(s)\,ds\\
    &\in \bigcup_{\xi\in\U}\Big[\Psi(t,t_0)\rho\int_{-\infty}^{t_0}\Psi(t_0,s)\zeta(s)\,ds+\rho\int_{t_0}^t\Psi(t,s)\xi(s)\,ds\Big]\\
    &=\Phi^{t,t_0}\bigg(\rho\int_{-\infty}^{t_0}\Psi(t_0,s)\zeta(s)\,ds\bigg)\subset\Phi^{t,t_0}(A_{t_0}).
    \end{split}
    \]

    We shall now prove the pullback and forward attractivity. Let $x_0\in\R^d$. Then, the definition of $\Phi^{t,s}$ and the variation of constants formula yield,
    \[
    \begin{split}
    h(\Phi^{t,t_0}(x_0), A_t)&= h\bigg(\bigcup_{\xi\in\U}\Big[\Psi(t,t_0)x_0 +\rho\int_{t_0}^t\Psi(t,s)\xi(s)\,ds\Big], \bigcup_{\xi\in\U}\rho\int_{-\infty}^t\Psi(t,s)\xi(s)\,ds\bigg)\\
    &=\sup_{\xi\in\U}\inf_{\zeta\in\U}\bigg|\Psi(t,t_0)x_0 +\rho\int_{t_0}^t\Psi(t,s)\xi(s)\,ds-\rho\int_{-\infty}^t\Psi(t,s)\zeta(s)\,ds\bigg|\\
    \end{split}
    \]
    Then the forward convergence is due to the fact that, for each $\xi\in\U$, $x_\xi$ is exactly the unique hyperbolic attracting solution of the inhomogeneous linear problem \eqref{eq:linearNAODEwNOISE}.
    In order to prove the pullback convergence, note that 
    \[
    \begin{split}
    &\sup_{\xi\in\U}\inf_{\zeta\in\U}\bigg|\Psi(t,t_0)x_0 +\rho\int_{t_0}^t\Psi(t,s)\xi(s)\,ds-\rho\int_{-\infty}^t\Psi(t,s)\zeta(s)\,ds\bigg|\\
    &\qquad\qquad\le |\Psi(t,t_0)|\,|x_0|+\sup_{\xi\in\U}\inf_{\zeta\in\U}\bigg|\rho\int_{t_0}^t\Psi(t,s)\xi(s)\,ds-\rho\int_{-\infty}^t\Psi(t,s)\zeta(s)\,ds\bigg|
    \end{split}
    \]
    and the right-hand side of the previous inequality tends to zero as $t_0\to\infty$ due to exponential stability of \eqref{eq:linearNAODE}. 
    
    We now show that $ \A$ is uniformly bounded. Note that for any $\xi\in\U$, any solution of \eqref{eq:linearNAODEwNOISE} satisfies 
    \[
    \begin{split}
    \bigg|\Psi(t,t_0)x_0 +\rho\int_{t_0}^t\Psi(t,s)\xi(s)\,ds\bigg|
    &\le |x_0|Ke^{-\gamma(t-t_0)}+\rho Ke^{-\gamma t}\int_{t_0}^te^{\gamma s}\sup_{\xi\in\U}|\xi(s)|\, ds \\  
    &\le |x_0|Ke^{-\gamma(t-t_0)}+\frac{\rho K \big(1-e^{-\gamma(t-t_0)}\big)}{\gamma},
    \end{split}
    \]
    and this holds uniformly in $\xi\in \U$. Therefore, for every bounded set $D\subset \R^d$, there is $T(D) >0$ such that if $t-t_0>T(D)$ then $h(\Phi^{t,t_0} (x_0), B_r(0))=0$ for any $x_0\in D$, where $r=1+\rho K/ \gamma>0$. Hence, it must be that $A_t\subset B_{r}(0)$ for all $t\in\R$, that is, $\A$ is uniformly bounded. \smallskip

 Finally, we show that     $\A$ is minimal with respect to the property of being invariant. Let $\mathcal{S}={S_t}$ be a nonempty closed invariant subset of $\A$ and assume by contradiction that $\mathcal{S}\subsetneq \A$. Then, there must exist $t_0\in\R$ and a function $\zeta\in L^\infty (\R,\R^d)$ with $\|\zeta\|_\infty\le 1$ exists such that 
\[
x_\zeta(t)=\rho\int_{-\infty}^{t}\Psi(t,s)\zeta(s)\,ds\notin S_t\qquad \text{for all } t<t_0,
\]
otherwise the invariance of $\mathcal{S}$ would already be contradicted. Moreover, since $\mathcal{S}$ is closed, there exists $\ep>0$ such that 
\begin{equation}\label{eq:21/01-17:22}
B_\ep(x_\zeta(t_0))\cap S_{t_0}=\varnothing.
\end{equation}
Consider $s<t_0-\frac{1}{\gamma}\log(\frac{2Kr}{\ep})$ where the constants $K\ge1$ and $\gamma>0$ are due to the exponential stability of $\dot x=L(t)x$ and the constant $r>0$ is the radius of the ball which contains every fiber of $\A$, thanks to the uniform boundedness shown above. Then, we have that, taking any point $p\in S_s$,
\[
|x_\zeta(t_0)-x(t,s,p,\zeta)|\le|\Psi(t_0,s)||x_\zeta(s)-p|\le K e^{-\gamma(t_0-s)}|x_\zeta(s)-p|<\ep.
\]
However, this means that $\Phi(t_0,s)p\cap B_\ep(x_\zeta(t_0))\ne \varnothing$, and thus $\Phi(t_0,s)S_s\cap B_\ep(x_\zeta(t_0))=S_{t_0}\cap B_\ep(x_\zeta(t_0))\ne\varnothing$, which contradicts \eqref{eq:21/01-17:22}. Therefore, the set $\A$ must be minimal invariant and the proof is complete.
\end{proof}

\begin{rmk}
    The assumption of exponential stability for \eqref{eq:linearNAODE} is equivalent to requesting that the dichotomy spectrum of $L(t)$ is strictly contained in the negative half-line. 
\end{rmk}

Next, we investigate the structure of the attractor's fibres studying various geometric and topological properties of the attractor. 

\begin{cor}[Symmetry with respect to the origin]\label{cor:symmetry}
     Assume that~\eqref{eq:linearNAODE} is exponentially stable and $U\subset\R^d$ is symmetric with respect to the origin. Then, the set $A_t$ obtained from \eqref{eq:linearNAODEwNOISE} with $\xi(t)\in U$ for a.e.~$t\in\R$ is symmetric with respect to the origin for every $t\in\R$.
\end{cor}
\begin{proof}
    The result is a direct consequence of \eqref{eq:hyp_sols_lin} and the observation that, if $\xi(t)\in U$ then also $-\xi(t)\in U$.
\end{proof}

\begin{cor}[The attractor scales linearly with $\rho$] \label{cor:scale}
     Assume that~\eqref{eq:linearNAODE} is exponentially stable. Then, $A_t$ scales linearly with $\rho$ for all $t\in\R$.
\end{cor}
\begin{proof}
    The result is an immediate consequence of the linearity in the parameter $\rho$ in \eqref{eq:hyp_sols_lin}.
\end{proof}

\begin{cor}[Convexity of the attractor's fibres]\label{cor:convexity}
    Assume that~\eqref{eq:linearNAODE} is exponentially stable.  Then, $\A$ is (fibre-wise) convex.
\end{cor}
\begin{proof}
    Let us fix $t\in\R$, and consider $x,y\in A_t$ and the point $\alpha\,x+(1-\alpha) y$, with $\alpha\in [0,1]$. Due to \eqref{eq:hyp_sols_lin}, one has
    \[
    \alpha\,x+(1-\alpha) y= \rho\int_{-\infty}^t\Psi(t,s)\big(\alpha\,\xi_x(s)+(1-\alpha)\xi_y(s)\big)\,ds
    \]
    and since $|\alpha\,\xi_x(s)+(1-\alpha)\xi_y(s)|\le \alpha\,|\xi_x(s)|+(1-\alpha)|\xi_y(s)|\le 1$ for every $s\in \R$, then $\alpha\,\xi_x(\cdot)+(1-\alpha)\xi_y(\cdot)\in\U$. Hence, $\alpha\,x+(1-\alpha) y\in A_t$. Every fibre of $\A$ is therefore convex. 
\end{proof}

\begin{prop}[Closedness of the attractor's fibres]
     Assume that~\eqref{eq:linearNAODE} is exponentially stable.  Then, $\A$ is (fibre-wise) closed.
\end{prop}
\begin{proof}
The proof of this result closely follows arguments similar to the ones in Theorem 1A at page 164 of \cite{lee1986foundations}. On this account, we only highlight the most important changes in the simple case when a supporting hyperplane exists which intersects the boundary in a fibre in a unique point. Namely, fix $\tau \in\R$  and  consider $P_0\in\partial A_{\tau}$. 
    Suppose that there is a  supporting hyperplane $H$ to $\overline A_{\tau}$ at $P_0$, so that $H\cap \overline A_{\tau}=\{P_0\}$. 
    Let $n$ be the unit normal vector orthogonal to $H$ pointing outwards and consider the initial value problem for the adjoint system 
    \[
    \dot \eta =-\eta L(t),\qquad \eta(\tau)=n^\top,
    \]
    which yields the globally defined unique (row) solution $\eta(t)=n\Psi(t,\tau)^{-1}$. Recalling that Pontryagin Maximum Principle is sufficient and necessary for linear control systems such as \eqref{eq:linearNAODEwNOISE} with a convex cost function such as $f_0$ as defined in \eqref{eq:cost-function}, we have that a unique optimal control $\overline u(t)$ exists satisfying
    \[
    \eta(t)\overline u(t)=\max_{u\in \overline B_1(0)}\eta(t)u, \quad \text{for a.a. }t\in\R.
    \]
    In particular, since $u(t)\in \overline B_1(0)$ for all $t\in\R$, then it must be that $|\overline u(t)|=1$ for almost all $t\in\R$. Let $\overline x(t)$ be the solution to $\dot x =L(t)x+\rho\overline u(t)$ defined by
    \[
    \overline x(t)=\rho\int_{-\infty}^t\Psi(t,s)\overline u(s)\, ds,\quad t\in\R.
    \]
    Then, one has that (see also Lemmas 2A and 3A in \cite{lee1986foundations}), 
    \[
    \eta(t)\overline x(t)=\max_{x\in A_t}\eta(t) x=\max_{x\in \overline A_t}\eta(t) x,\quad \text{for all } t\in\R.
    \]
    Hence, $\overline x(\tau)\in H\cap\overline  A_\tau=\{P_0\}$, and thus $\overline x(\tau)=P_0\in A_\tau $. For the case when  $P_0\in\partial A_t$ fails to be the unique intersection of a supporting hyperplane with $A_\tau$ we refer the reader to the second part of the proof in Theorem 1A at page 164 of \cite{lee1986foundations} whose arguments can be followed exactly.
\end{proof}

\begin{prop}
    Assume that~\eqref{eq:linearNAODE} is exponentially stable. Then, for any $t\in\R$, no two distinct points in $\partial A_t$ admit the same outward unit normal vector.
\end{prop}
\begin{proof}
    Assume on the contrary that there are $\tau\in \R$ and  $\overline x_1, \overline x_2\in \partial A_\tau $, $\overline x_1\neq \overline x_2$ with the same unit normal vector $n\in\mathbb S^{d-1}$, i.e.~$n\, \bot\, T_{\overline x_1}\partial A_\tau$ and $n\, \bot\, T_{\overline x_2} \partial A_\tau$. Fix $t_0<\tau$ and $x_0\in\partial A_{t_0}$, and let $\eta(t)$ be the unique solution of the initial value problem
    \[
    \dot \eta =-\eta L(t), \quad \eta(\tau)=n.
    \]
    Recalling that Pontryagin Maximum Principle is sufficient and necessary for linear control systems such as \eqref{eq:linearNAODEwNOISE} with a convex cost function such as $f_0$ as defined in~\eqref{eq:cost-function}, we have that a unique extremal control $\overline u(t)$ exists satisfying
    \[
    \eta(t)u(t)=\max_{u\in B_1}\eta(t)u, \quad \text{for a.a. }t\in[t_0,t].
    \]
    Let $t_0<\tau$ and 
    \[
    x_0=\rho\int_{-\infty}^{t_0} \Psi(t_0,s)u(s)\,ds.
    \]
    Then, $x_0\in \partial A_{t_0}$ (see Theorem \ref{thm:exist-attractor} and the response functions $x_{1,2}(t)$ satisfying $x_1(\tau)=\overline x_1\neq\overline x_2=x_2(\tau)$ solve the initial value problem
    \[
    \dot x =L(t)x+\rho u(t),\quad x(t_0)=x_0\in\partial A_{t_0}.
    \]
    Due to the uniqueness of solutions, we have that $x_1(t)=x_2(t)$ for all $t\in[t_o,\tau]$, which is a contradiction and completes the proof.  
\end{proof}

\begin{cor}[Strict convexity of the attractor's fibres]\label{cor:linear-strictly-convex}
    Assume that~\eqref{eq:linearNAODE} is exponentially stable. Then, $A_t$ is strictly convex for every $t\in\R$.
\end{cor}

\subsection{Boundary system for continuous nonautonomous linear systems}
Assume that  the linear homogeneous system \eqref{eq:linearNAODE} is exponentially stable, and denote by $\A$ the minimal invariant set of the set-valued system of \eqref{eq:linearNAODEwNOISE} provided by Theorem \ref{thm:exist-attractor}. Then, Theorem \ref{thm:BM-NAODE-inv}\textit{(i)} assures that the unit normal bundle of $\A$ is backward invariant with respect to the boundary system,
\begin{equation}\label{eq:linear-boundary-map}
\begin{cases}
    \dot x=L(t)x+\rho\,n,\\
    \dot n=-L(t)^\top n+\langle n,L(t)^\top n\rangle\, n.
\end{cases}
\end{equation}
To show the invariance of the outward unit normal bundle we can proceed in two ways: we can prove it directly using the fact that the Pontryagin Maximum Principle is both necessary and sufficient in the linear case, or we can prove the differentiability of the boundary of $\A$ and then apply  Theorem~\ref{thm:BM-NAODE-inv}\textit{(ii)}. We take the second route, convinced that the differentiability of the boundary is an interesting property by itself. Notably, it means that the solution of the linear Hamilton-Jacobi equation are classical solutions and not viscosity solutions, a property that should be widely known in the optimal control literature but that seems hard to find. We remark that the backward invariance with respect to the boundary system \eqref{eq:linear-boundary-map} is essential in the proof.

\begin{thm}[Differentiability of the attractor's boundary]\label{thm:linear-boundary-C1}
    Assume that~\eqref{eq:linearNAODE} is exponentially stable. Then, $\partial A_t$ is a $C^1$-differentiable manifold for every $t\in\R$.
\end{thm}
\begin{proof}

  It suffices to show that for all $t\in\R$ and all $p\in \partial A_t$ there is associated a unique unit vector $n\in N_{1,p}^+\partial A_t$ which is therefore the normal to the boundary of the reachable set $A_t$ at $p$. Suppose on the contrary that there exist $\tau \in \mathbb{R}$ and $p\in\partial A_\tau$  such that there are two distinct normal vectors $n_1\ne n_2$ in $N_{1,p}^+\partial A_{\tau}$. Then, consider the initial value problems for the adjoint system 
    \begin{equation}\label{eq:07/01-17:36}
    \dot \eta =-\eta L(t),\qquad \eta(t_0)=n_i^\top, \qquad i=1,2.
 \end{equation}
    which yield the globally defined unique (row) solutions $\eta_i(t)=n_i\Psi(t,t_0)^{-1}$, for $i=1,2$. Note also that the uniqueness of solutions yields $\eta_1(t)\neq\eta_2(t)$ for all $t\in\R$. Recalling that Pontryagin Maximum Principle is sufficient and necessary for linear control systems such as~\eqref{eq:linearNAODEwNOISE} with a convex cost function such as $f_0$ as defined in~\eqref{eq:cost-function}, we have that two optimal controls $\overline u_i(t)$ exist satisfying
    \begin{equation}\label{eq:31/03-18:33}
    \eta_i(t)\overline u_i(t)=\max_{u\in \overline B_1(0)}\eta_i(t)u, \quad \text{for a.a. }t\in\R, \ i=1,2.
    \end{equation}
    In particular, $\overline u_1(t)\neq \overline u_2(t)$ for almost all $t\in\R$, and since $u_i(t)\in \overline B_1(0)$ for all $t\in\R$, then it must be that $|\overline u_i(t)|=1$ for almost all $t\in\R$. Let $\overline x_i(t)$ be the solutions to $\dot x =L(t)x+\rho\overline u_i(t)$ defined by
    \[
    \overline x_i(t)=\rho\int_{-\infty}^t\Psi(t,s)\overline u_i(s)\, ds,\quad t\in\R.
    \]
Note also that  $\overline x_i(t)$, for $i=1,2$, are the projections in $\mathbb{R}^d$ of the backward solutions of the boundary flow, that is the normalised Pontryagin Hamilton equations \eqref{eq:linear-boundary-map} with initial conditions $(p,n_i)$ for $i=1,2$ at time $\tau$, and thus $p=\overline x_1(\tau)=\overline x_2(\tau)$. On the other hand, notice that there is at least one $t_0<\tau$, and by continuity at least an open interval around $t_0$, such that $\overline x_1(t_0)\ne \overline x_2(t_0)$ for if not, we would have that
  \[0=\dot{x}_1-\dot{x}_2=L(t)(x_1(t)-x_2(t))+\rho (n_1(t)-n_2(t))=\rho (n_1(t)-n_2(t)),\]
  leading to $n_1(t)=n_2(t)$, and thus to $\overline u_1(t)=\overline u_2(t)$ for all $t< \tau$, which contradicts the uniqueness of solutions for the adjoint system \eqref{eq:07/01-17:36} due to \eqref{eq:31/03-18:33}.
  
  Since the corresponding solutions $\overline x_1(t_0)\ne \overline x_2(t_0) \in \partial A_{t_0}$ are distinct, it follows by linearity that the averaged curve $\tilde{x}(t)=\frac{\overline x_1(t)+\overline x_2(t)}{2}$ is also a solution of the system of equations:
  \[\dot{\tilde{x}}(t)=L(t)\tilde{x}+\tilde{u}(t)\]
  with respect now to the averaged controls $\tilde{u}(t)=\frac{\overline u_1(t)+\overline u_2(t)}{2}\in \overline{B_1(0)}$. But since the terminal point $\tilde{x}(\tau)=\frac{\overline x_1(\tau)+\overline x_2(\tau)}{2}=\frac{p+p}{2}=p\in \partial A_{\tau}$ is also reachable by the averaged solution at time $t=\tau$, it follows by the maximum principle again that the averaged control $\tilde{u}(t)$ should also be optimal, i.e. that $\tilde{u}(t)\in \partial \overline{B_1(0)}=\mathbb{S}^{d-1}$ for all $t<\tau$. But this is impossible since the inequality $u_1^*(t)\ne u_2^*(t)$ for a.a.~$t<\tau$, implies that $||\tilde{u}(t)||<1$. This contradicts the maximum principle for $\tilde x(\tau)\in \partial A_{\tau}$, and finishes the proof of the theorem. 
\end{proof}

\begin{cor}\label{thm:BM-NAODE-linear}
    Assume that~\eqref{eq:linearNAODE} is exponentially stable. The outward unit normal bundle $N^+_1\partial \A$ is invariant with respect to the boundary system \eqref{eq:linear-boundary-map}.
\end{cor}

An important consequence of Corollary \ref{thm:BM-NAODE-linear} is that \eqref{eq:linear-boundary-map} can be studied using a classic technique from nonautonomous dynamical systems theory called skew-product formalism. In particular, this approach allows us to show that the outward unit normal bundle is \emph{attractive} as we will clarify later.
\smallskip

In order to carry out the construction of the flow, we recall some notation and definitions. Given a function $g:\R\times \R^N\to\R^N$, with $N\in\N$, and some $\tau \in\R$, we call the time-translation of $g$ at time $\tau$, the function 
\[
g_\tau:\R\times \R^N\to\R^N,\qquad (t,y)\mapsto g_\tau(t,y)=g(t+\tau,y). 
\]
Moreover, the closure of the set of time-translations of $g$ with respect to a suitable topology, is called the hull of $g$, 
\[
\Omega_{g}=\mathrm{cls}\{g_t\mid  t\in\R\}.
\]
The employed topology plays a fundamental role in guaranteeing the continuity of the skew-product flow induced by the differential equation $\dot y =g(t,y)$, i.e.~the function defined by
\[
(\sigma,\varphi): \R\times \Omega_{g}\times \R^N\to\Omega_{g}\times \R^N,\qquad (t,\omega,y_0)\mapsto\big(\omega_t, \varphi(t,\omega,y_0)\big),
\]
where $\varphi(t,\omega,y_0)$ represents the solution of the initial value problem $\dot y= \omega(t,y)$, with $y(0)=y_0$. 
For Carath\'eodory differential equations like the ones considered in this work, optimal topologies to guarantee the continuity of the induced skew-product flow are presented in \cite{longo2018topologies,longo2017topologies,longo2019weak}.
\smallskip

In particular, \eqref{eq:linear-boundary-map} induces an (autonomous) skew-product flow on the augmented phase space $\Omega_{L}\times \R^{2d}$ via the map
\begin{equation}\label{skew-product-linear} 
\begin{split}
(\sigma,\varphi):\R^+\times\Omega_L\times\R^d\times \mathbb S^{d-1}\to \Omega_L\times\R^d\times \mathbb S^{d-1},\\
(t,\omega,x_0,n_0)\mapsto\big(\omega_t, \varphi(t,\omega,x_0,n_0)\big),
\end{split}
\end{equation}
where $\varphi(t,\omega,x_0,n_0)=\big(x(t,\omega,x_0,n_0),n(t,\omega,x_0,n_0)\big)$, denotes the solution of \eqref{eq:linear-boundary-map} with initial data $(x_0,n_0)$ at time zero.\smallskip

The above skew-product flow is also induced by the following more general functional system 
\begin{equation}\label{eq:linear-boundary-map-skew}
\begin{cases}
    \dot x=\mathscr{L}(\omega_t)x+\rho\,n,\\
    \dot n=-\mathscr{L}(\omega_t)^\top n+\langle n,\mathscr{L}(\omega_t)^\top n\rangle\, n,
\end{cases}
\end{equation}
where 
\[
\mathscr{L}:\Omega_L\to \R^{d\times d},\quad \omega\mapsto \mathscr{L}(\omega)= \omega(0).
\]
Note that, since $\mathscr{L}(\omega_t)= \omega_t(0)=\omega_t$, \eqref{eq:linear-boundary-map-skew} reduces to \eqref{eq:linear-boundary-map} when $\omega=L\in\Omega_L$. \smallskip

We also recall that a family $\widehat \A=\{ A_\omega\mid \omega\in \Omega_L\}$ of nonempty, compact sets of $\R^d\times \mathbb S^{d-1}$ is said to be a \emph{pullback attractor for the skew-product semiflow} \eqref{skew-product-linear}, if it is invariant, i.e.
\[ \varphi(t,\omega,A_\omega)=A_{\omega_t} \quad \text{ for each } t\ge 0 \;\text{ and }\; \omega\in \Omega_L\,,\]
and, for every nonempty bounded set $D$ of $\R^d\times \mathbb S^{d-1}$ and every $\omega\in \Omega_L$ one has
\begin{equation*}
\lim_{t\to\infty} \dist(\varphi(t,\omega_{-t},D), A_\omega)=0\,,
\end{equation*}
where $\dist(A,B)$ denotes the \emph{Hausdorff semi-distance} of two nonempty sets $A$, $B$ of~$\R^d\times \mathbb S^{d-1}$.  A pullback attractor for the skew-product flow is \emph{bounded} if
\begin{equation*}
\bigcup_{\omega\in \Omega_L} A_\omega\quad\text{is bounded.}
\end{equation*} 
Moreover, a compact set $\mathbb A$ of $\Omega_L\times\R^d\times \mathbb S^{d-1}$ is said to be a \emph{global attractor for the skew-product flow} \eqref{skew-product-linear}, if it is the maximal nonempty compact subset of $\Omega_L\times\R^d\times \mathbb S^{d-1}$ which is invariant, i.e.
\begin{equation*}
 (\sigma,\varphi)(t,\mathbb A)=\mathbb A \quad \text{ for each }\; t\ge 0\,,
\end{equation*}
and attracts all compact subsets $\mathcal D$ of $\Omega_L\times\R^d\times \mathbb S^{d-1}$, i.e.
\[\lim_{t\to\infty} \dist( (\sigma,\varphi)(t,\mathcal D),\mathbb A)=0\,,\]
where now $\dist(\mathcal B,\mathcal C)$ denotes the \emph{Hausdorff semi-distance} of two nonempty sets $\mathcal B$, $\mathcal C$ of $\Omega_L\times\R^d\times \mathbb S^{d-1}$.

We can now show that the outward unit normal bundle is the unique pullback attractor for the skew-product flow induced by the boundary system \eqref{eq:linear-boundary-map-skew}. Moreover, a global attractor for the skew-product flow exists is the base flow is compact and the outward unit normal bundle corresponds to the projection in phase-space of the global attractor.

\begin{thm}[Attractivity of the outward unit normal bundle]\label{thm:attractivity-UNB}
Assume that the linear homogeneous system $\dot x = L(t)x$ is exponentially stable, and that the skew-product flow
induced by \eqref{eq:linear-boundary-map} is continuous. Then, the following statements hold true.
\begin{itemize}
\item[(i)] For every $\omega\in\Omega_L$ there exists a uniform attractor $\A^\omega=(A^\omega_t)_{t\in\R}$ for
\begin{equation}\label{eq:linear_ev-funct}
    \dot x = \mathscr{L}(\omega_t)x+\rho \xi(t),\quad \text{for } \xi\in\U.
\end{equation}

\item[(ii)] A bounded set $B\subset \R^d$ exists such that 
\[
\bigcup_{\omega\in\Omega_L,\, t\in\R} A^\omega_t\subset B.
\]

\item[(iii)] Let $A_\omega=A^\omega_0$ for all $\omega\in\Omega_L$, and define $\widehat\A=\{A_\omega\mid\omega\in\Omega_L\}$ and $N^+_1\partial \widehat\A=(N^+_1\partial A_\omega)_{\omega\in\Omega_L}$. The outward unit normal bundle $N^+_1\partial \widehat\A$ is the unique pullback attractor for the skew-product flow~\eqref{skew-product-linear},  and for any $\omega\in\Omega_{L}$ we have 
    \begin{equation}\label{eq:thm-attract-i}
        N^+_1\partial A_\omega=\bigcap_{t\ge 0}\; \overline{\bigcup_{\tau\ge t}\varphi(\tau,\omega_{-\tau}, B\times  \mathbb S^{d-1})}.
        \end{equation}
\item[(iv)] If $\Omega_{L}$ is compact, there is a global attractor of the skew-product flow~\eqref{skew-product-linear} given by
    \[ \mathbb A=\bigcap_{t\ge 0}\; \overline{\bigcup_{\tau\ge t} (\sigma,\varphi)(\tau, \Omega_{L}\times B\times \mathbb S^{d-1})}\,
    =\bigcup_{\omega\in \Omega_{L}}\{ \omega\}\times N^+_1\partial A_\omega\,.\]
\end{itemize}
\end{thm}

\begin{proof}
Firstly, note that for all $\omega\in\Omega_L$
\[
\frac d{dt}|n|^2= -2\langle n,\mathscr{L}(\omega_t)^\top n\rangle+2\langle n,\mathscr{L}(\omega_t)^\top n\rangle |n|^2
\]
and therefore  the unit sphere $\mathbb{S}^{d-1}$ (and thus also the unit ball $B_1(0)$) is invariant with respect to the nonlinear equations 
\begin{equation}\label{eq:linear-adjoint-skew}
\dot n=-\mathscr{L}(\omega_t)^\top n+\langle n,\mathscr{L}(\omega_t)^\top n\rangle\, n,\quad \text{for all }\omega\in\Omega_L.
\end{equation}
We can therefore focus on the equations 
\[
    \dot x = \mathscr{L}(\omega_t)x+\rho \xi(t),\quad \text{for }\omega\in\Omega_L\text{ and } \xi\in\U.
\]
By assumption, there are constants $K\ge 1$ and $\gamma>0$ such that the fundamental matrix solution of $\dot x = L(t)x$ satisfies $|\Psi(t,t_0)|\le Ke^{-\gamma (t-t_0)}$. Then, due to the continuity of the skew-product flow \eqref{skew-product-linear} we have that the fundamental matrix solution of $\dot x = \mathscr{L}(\omega_t)x$ for $\omega\in\Omega$ also satisfies $|\Psi_\omega(t,t_0)|\le Ke^{-\gamma (t-t_0)}$ for the same constants $K\ge 1$ and $\gamma>0$---see for example \cite[Proposition 4.3(iii)]{longo2019weak} for a proof of a more general result guaranteeing the propagation of an exponential dichotomy for Carath\'eodory linear equations.

Then, Theorem \ref{thm:exist-attractor} provides the existence of a uniform attractor $\mathcal A^\omega= (A^\omega_t)_{t\in\R}$---defined by \eqref{eq:hyp_sols_lin}---for every set-valued system $\dot x\in \overline B_\rho(\omega(t)x)$ with $\omega\in\Omega_L$ and completes the proof of (i). \par\smallskip

Note also that for any $t\in\R$,
\[
\sup_{\omega\in\Omega_L, \xi\in\U}\Big|\rho\int_{-\infty}^t\Psi_\omega(t,s)\xi(s)\,ds\Big|\le 
\frac{\rho K} \gamma.
\]
Hence, the set $B=B_{\frac{\rho K} \gamma +1}(0)\subset \R^d$ satisfies (ii).\par\smallskip

Let $A_\omega=A^\omega_0$ for all $\omega\in\Omega_L$, and define $\widehat\A=\{A_\omega\mid\omega\in\Omega_L\}$ and $N^+_1\partial \widehat\A=(N^+_1\partial A_\omega)_{\omega\in\Omega_L}$. 
For any $\omega\in\Omega_L$ and $x_0\in\partial A_{\omega}$, the Maximum Principle for \eqref{eq:linear_ev-funct}  with $\mathscr L$ evaluated at $\omega$ is satisfied by the solution $(\overline x (t),\overline n(t))$ of the boundary system \eqref{eq:linear-boundary-map-skew} with $\mathscr L$ evaluated at $\omega$ and initial data $(x_0, n_0)\in N^+_1\partial A_{\omega}$.

For any $\tau\in\R$ and $x_0\in\partial A_{\omega_{-\tau}}$, the Maximum Principle for \eqref{eq:linear_ev-funct} at $\omega_{-\tau}$ is satisfied by the solution $(\overline x (t),\overline n(t))$ of the boundary system \eqref{eq:linear-boundary-map-skew} at $\omega_{-\tau}$ with initial data $(x_0, n_0)\in N^+_1\partial A_{\omega_{-\tau}}$.
This fact has the following implications. Firstly,
\[
n(\tau,\omega_{-\tau}, \overline n(-\tau))=n_0,\quad\text{for all }\tau>0,
\]
and therefore $n(\tau,\omega_{-\tau}, \mathbb{S}^{d-1})=\mathbb{S}^{d-1}$. Secondly,
$$
\overline x(t)=\Psi_{\omega}(t,t_0)\overline x(t_0)+\rho\int_{t_0}^t \Psi_{\omega}(t,s)\overline n(s)\, ds\quad \text{for all }t_0<t,
$$
which implies 
\[
\begin{split}
|\overline x(t)-x_{\overline n}(t) |&\le |\Psi_{\omega}(t,t_0)\overline x(t_0)|+ \Big|\int_{-\infty}^{t_0}\Psi_{\omega}(t,s)\overline n(s)\, ds \Big|\\
&\le \left(\frac{\rho K} \gamma +1\right)Ke^{-\gamma (t-t_0)} +\Big|\int_{-\infty}^{t_0}\Psi_{\omega}(t,s)\overline n(s)\, ds \Big|,
\end{split}
\]
where the right hand-side of the previous inequality  is obtained using (i) and (ii), and it can be made as small as desired by taking $t_0$ as small as needed. Hence, $\overline x(t)=x_{\overline n}(t)$ for all $t\in\R$, that is, $\overline x$ coincides with the unique hyperbolic solution of \eqref{eq:linear_ev-funct} with $\mathscr L$ evaluated at $\omega$ and with $\xi=\overline n$. Then, we have that for any $p\in B$,
\[
\begin{split}
| x\big(\tau,\omega_{-\tau},\, &p,\overline n(-\tau)\big)- \overline x(0)|\\
&=| \Psi_{\omega_{-\tau}}(\tau,0)p +\int_{0}^{\tau}\Psi_{\omega_{-\tau}}(\tau,s)\overline n(s-\tau)\, ds -\int_{-\infty}^{0}\Psi_{\omega}(0,s)\overline n(s)\, ds|\\
&\le |p|Ke^{-\gamma \tau}+\Big| \int_{0}^{\tau}\Psi_{\omega}(\tau-\tau,s-\tau)\overline n(s-\tau)\, ds -\int_{-\infty}^{0}\Psi_{\omega}(0,s)\overline n(s)\, ds\Big|\\
&=|p|Ke^{-\gamma \tau}+\Big| \int_{-\tau}^{0}\Psi_{\omega}(0,u)\overline n(u)\, du -\int_{-\infty}^{0}\Psi_{\omega}(0,s)\overline n(s)\, ds\Big|\\
&=|p|Ke^{-\gamma \tau}+\Big| \int_{-\infty}^{-\tau}\Psi_{\omega}(0,s)\overline n(s)\, ds\Big|
\end{split}
\]
and the right hand-side of the previous chain of inequality can be made as small as desired choosing $\tau$ big enough. Then, the arbitrariness on $p\in B$ implies that 
\[
\lim_{\tau\to\infty}\dist\Big(x\big(\tau,\omega_{-\tau},B,\overline n(-\tau)\big), x_{\overline  n}(0)\Big)=0.
\]
Gathering the previous information we have exactly \eqref{eq:thm-attract-i} which completes the proof of (iii).\par\smallskip

Finally, the existence of a global attractor $\mathbb A$ when  $\Omega_L$ is compact, follows from~Theorem 2.2 of Cheban~\emph{et al.}~\cite{cheban2002relationship} and, as shown in Theorem 16.2 of~\cite{carvalho2012attractors}, $N^+_1\partial A_\omega$ is the section of $\mathbb A$ over $\omega$, that is ${\mathbb A}=\bigcup_{\omega\in \Omega_L}\{ \omega\}\times N^+_1\partial A_\omega$, which finishes the proof of (iv) and of the theorem.
\end{proof}

\begin{example}\label{example}
    We aim to show the applicability of the above results by mean of a concrete example. Consider the two-dimensional nonautonomous linear system 
    \[
    \dot x= L(t)x=\begin{pmatrix}
        -20+10\arctan(0.1t) &(1+\arctan(t))\cos(0.1t)\\(1-\arctan(t))\sin(t) &-15+5\cos(0.5t)
    \end{pmatrix}x
    \]
    It is easy to show that this linear system is exponentially stable because it is row-dominant (cf.~\cite{fink2006almost} for the notion of row-dominance). Hence, Theorem \ref{thm:exist-attractor} guarantees the existence of a uniformly bounded forward and pullback attractor for the set valued system $\dot x\in\overline B_1(L(t)x)$. Note that we have fixed the size of the noise to $\rho=1$ (check Corollary \ref{cor:scale} for the implications to different values of $\rho$). 
       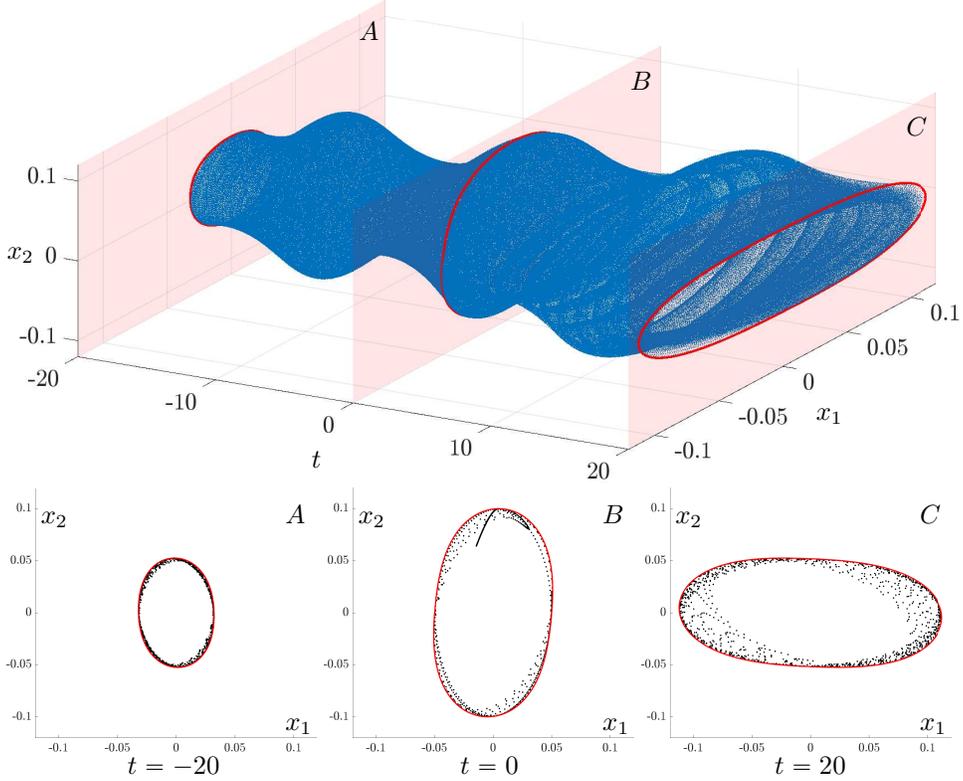
\begin{figure}[h]
    \centering
    \begin{overpic}[width=\textwidth]{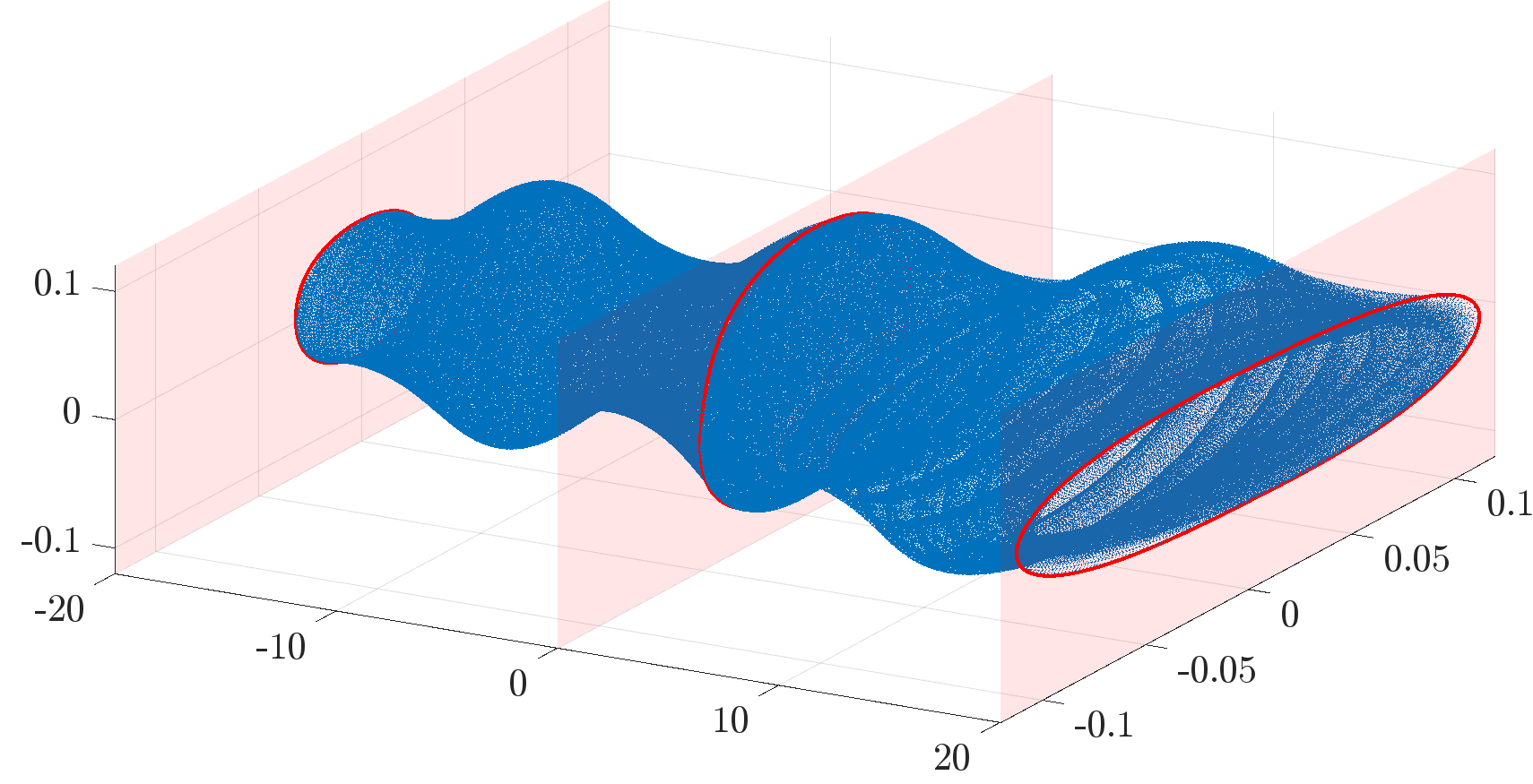}
    \put(32,2){$t$}
    \put(0,24){$x_2$}
    \put(85,7){$x_1$}
    \put(37,47){$A$}
    \put(65.5,41.7){$B$}
    \put(94.5,37){$C$}
    \end{overpic}
    \begin{overpic}[width=0.325\textwidth]{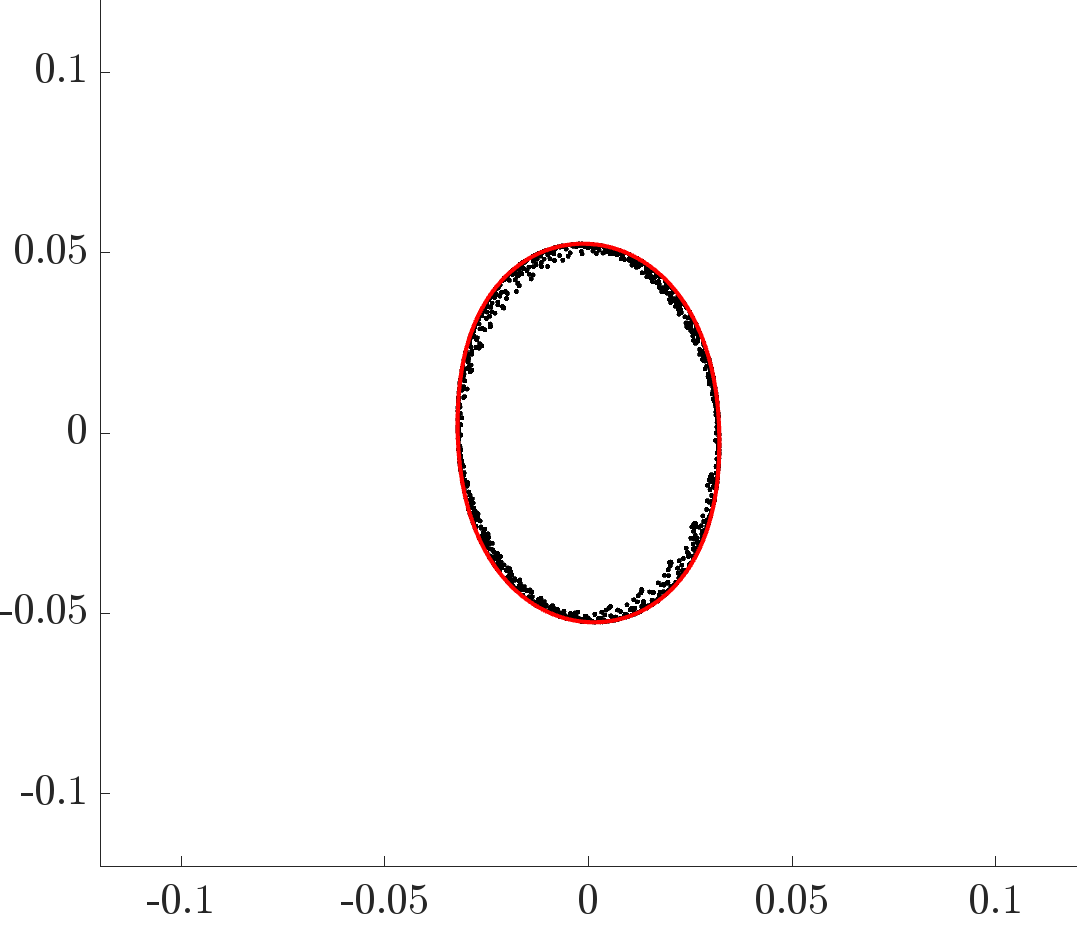}
    \put(90,75){$A$}
    \put(11,75){$x_2$}
    \put(90,8){$x_1$}
    \put(39,-6){$t=-20$}
    \end{overpic}
    \begin{overpic}[width=0.325\textwidth]{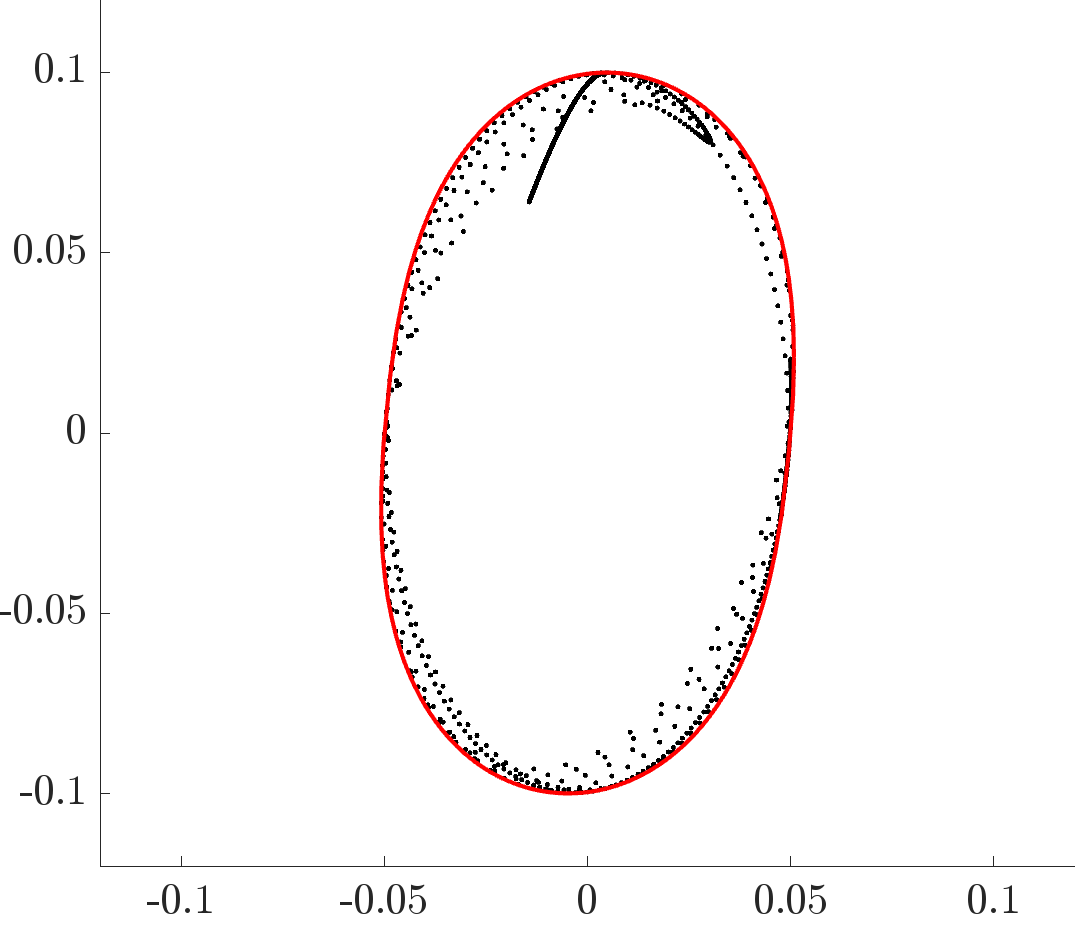}
    \put(90,75){$B$}
    \put(11,75){$x_2$}
    \put(90,8){$x_1$}
    \put(44,-6){$t=0$}
    \end{overpic}
    \begin{overpic}[width=0.325\textwidth]{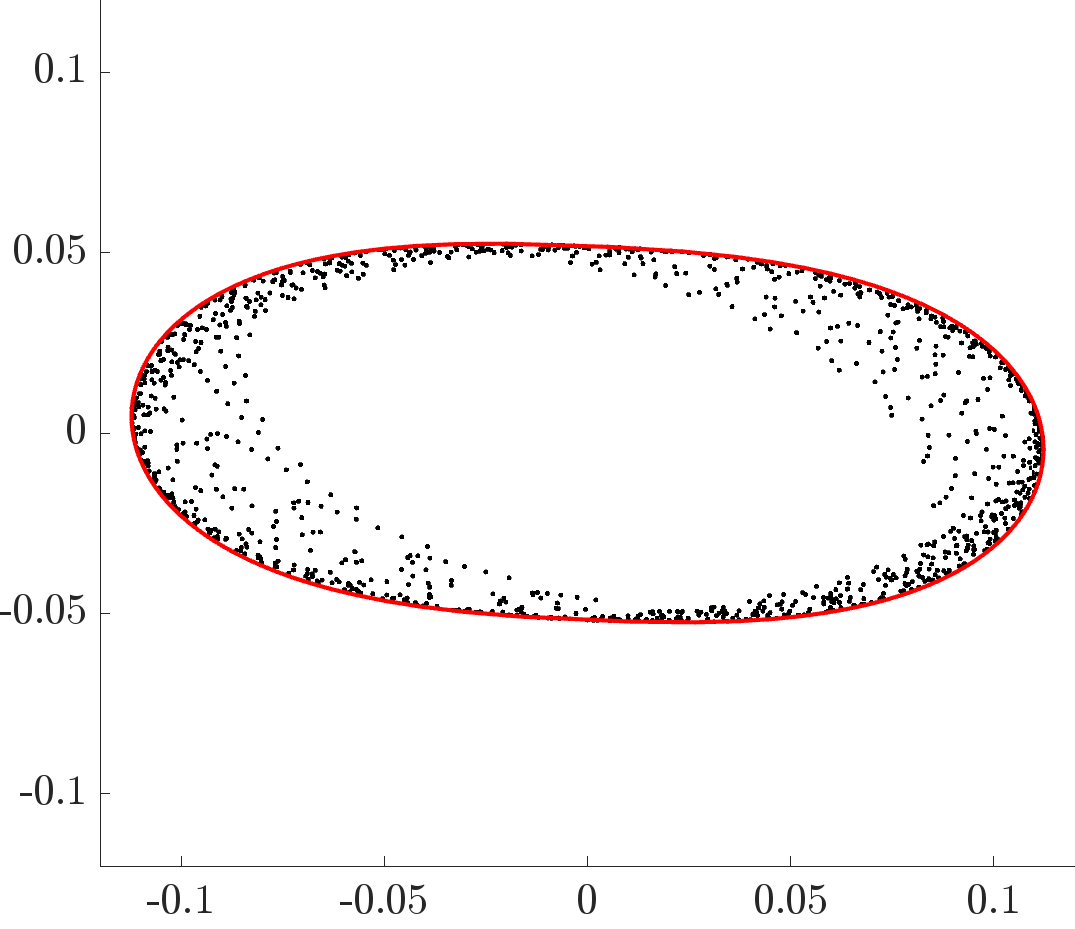}  
    \put(90,75){$C$}
    \put(11,75){$x_2$}
    \put(90,8){$x_1$}
    \put(43,-6){$t=20$}
    \end{overpic}
    \vspace{3mm}
    \caption{The upper panel shows an approximation of the boundary of the attractor for the set valued system defined in Example \ref{example} as a cloud of five-hundred trajectories started at the origin at $t_0=-100$ and different controls $u(t)$ with $|u(t)|=1$ for all $t\in\R$. A reduced number of controls and the constraint on their  norm were chosen to allow an intuitive representation of the volume and change of shape of the attractor in time.  The planes marked with the letters $A$ $B$ and $C$ section the attractor at times $t\in\{-20,0,20\}$. The red curves show the three sections of the boundary obtained via the boundary flow \eqref{eq:linear-boundary-map} projected to $R^2$.}
    \label{fig:attractor}
\end{figure} 
    The upper panel in Figure \ref{fig:attractor} shows an approximation of the boundary of the attractor in the interval $[-20,20]$ obtained via forward integration of five hundred trajectories starting at the origin at time $t_0=-100$ with respect to five-hundred different controls $u(t)$ with $|u(t)|=1$ for all $t\in\R$.  
    A reduced number of controls and the constraint on the norm of the controls have been chosen to provide an intuitive representation of the volume and change of shape of the attractor in time. 
    However, we know that the attractor is a (strictly) convex set which is symmetric with respect to the origin (see Corollaries \ref{cor:symmetry}, \ref{cor:convexity}, \ref{cor:linear-strictly-convex}). 
    It is not to be expected that all, or any, of the depicted five-hundred solutions lie on the boundary of the attractor as it is appreciable in the three lower panels which corresponds to the three sections of the three-dimensional plot at the times $t\in\{-20,0,20\}$.  
    The black dots are the points attained by the five hundred trajectories within each fibre. In each fibre, the boundary can be determined as the first component in the pullback limit in Theorem \ref{thm:attractivity-UNB}(iii). The red curves in Figure \ref{fig:attractor} show the three sections of the boundary. We use a special trick to make sure to obtain uniformly distributed points on the boundary. We firstly carry out a backward integration of \eqref{eq:linear-adjoint-skew} from $\tau\in\{-20,0,20\}$ back to $t_0=\tau-10$ of one-hundred-and-fifty initial conditions $\{n_i\mid i=1,\dots,150\}$ on the unit circle (corresponding to as many normal vectors), and then a joint forward integration until time $\tau$ of the solutions $(x_i(t),n_i(t))$ to \eqref{eq:linear-boundary-map} with initial conditions $\{(0,n(t_0,\tau,n_i))\mid i=1,\dots,150\}$ at time $t_0$. The red curve is made of the point $x_i(\tau)$ for $i=1,\dots,150$. Note also that $x_i(\tau)\in\partial A_\tau$ is exactly the point corresponding to the outward unit normal vector $n_i$.
\end{example}

\section{Conclusions and outlook}
This work provides a powerful framework to track the evolution of the boundary (in fact of the outward unit normal bundle) of a minimal invariant set for a nonautonomous differential inclusion induced by a differential equation with additive uniformly bounded noise, controls and/or other type of perturbations. The developed framework  allows to apply classic analytical and numerical techniques from dynamical systems theory to the deterministic single-valued differential system representing the evolution of the outward unit normal bundle. \par\smallskip

A natural question is whether the results in this work extend to the case where controls or perturbations take values in a convex set which is not necessarily the unit sphere or even in more general non-convex sets. While the fundamental connection with the Pontryagin Maximum Principle is still true, the construction of the boundary system must be carried out taking into account these changes. We believe that the characterisation of a broader class within which our results still hold true is an important problem for future work.    \par\smallskip

Our results also hold for autonomous differential equations, where we envision the stronger results hold thanks to the invariance with respect to time. \par\smallskip

We aim to push our accomplishments even further using this newly established approach to investigate ``topological and smooth bifurcations'' of minimal invariant sets in terms of the behaviour of their boundary. 
We have specifically worked in the nonautonomous context, driven by the societal relevance posed by tipping points (also known as critical transitions) under the time-dependent variation of parameters \cite{ashwin2012tipping}. Recent works \cite{duenas2023rate,duenas2023critical,duenas2024concave,duenas2024critical,kuehn2022estimating, longo2021rate, longo2024critical1, longo2024critical} have shown the undeniable relation between this type of tipping points and nonautonomous bifurcation theory~\cite{anagnostopoulou2023nonautonomous}. The boundary system in this work allows to rigorously investigate the impact of controls and uncertainties on tipping points in a generality never attempted before.\par\smallskip

\subsection*{Acknowledgments} The authors thank Gabriel Fuhrmann, Michal Fedorowicz, Wei Hao Tey, Jeroen Lamb, Satvik Sharma and Dima Turaev for insightful discussions which greatly helped the development of this paper and correction of its draft. Iacopo P.~Longo acknowledges partial support from UKRI under the grant agreement
EP/X027651/1, from  MICIIN/FEDER project PID2021-125446NB-I00 and from the University of
Valladolid under project PIP-TCESC-2020. The work of Konstantinos Kourliouros and Martin Rasmussen has been supported by the EPSRC grant EP/Y020669/1.  

\bibliographystyle{siam}
\addcontentsline{toc}{chapter}{Bibliography}
\bibliography{references}

\end{document}